\documentclass{tac}
\usepackage[all,ps,arc]{xy}  
\usepackage{amsmath,amssymb,latexsym}
\usepackage{pdfsync}
\usepackage{makeidx}
\usepackage[english]{babel}
\usepackage[applemac]{inputenc}
\usepackage{graphicx}
\usepackage{color}

\newtheorem{Theorem}{Theorem}[section]
\newtheorem{Lemma}[Theorem]{Lemma}
\newtheorem{Proposition}[Theorem]{Proposition}
\newtheorem{Definition}[Theorem]{Definition}

\newtheorem{Text}[Theorem]{}
\newcommand{\bA}{{\mathbb A}}
\newcommand{\bB}{{\mathbb B}}
\newcommand{\bC}{{\mathbb C}}
\newcommand{\bD}{{\mathbb D}}

\newcommand{\bF}{{\mathbb F}}

\newcommand{\bK}{{\mathbb K}}
\newcommand{\bP}{{\mathbb P}}
\newcommand{\bR}{{\mathbb R}}

\newcommand{\bX}{{\mathbb X}}

\newcommand{\cA}{{\mathcal A}}
\newcommand{\cB}{{\mathcal B}}

\newcommand\Grp{\mathbf{Grp}}

\newcommand\Grpd{\mathbf{Grpd}}

\newcommand\Ker{\mathrm{Ker}}
\newcommand\bKer{\mathbb K\mathrm{er}}
\newcommand\K{\mathrm{K}}

\newcommand\id{\mathrm{id}}
\newcommand\Id{\mathrm{Id}}

\title{The snail lemma for internal groupoids}
\author{S. Mantovani, G. Metere, E.M. Vitale}
\address
{Dipartimento di matematica, Universit\`a degli studi di Milano\\
Via C. Saldini 50, 20133 Milano, Italia.\\
Dipartimento di matematica e informatica, Universit\`a degli studi di Palermo\\
Via Archirafi 34, 90123 Palermo, Italia.\\
Institut de recherche en math\'ematique et physique, Universit\'e catholique de Louvain\\
Chemin du Cyclotron 2, B 1348 Louvain-la-Neuve, Belgique.}
\copyrightyear{2017}
\thanks{The first and the second author acknowledge the financial support of the I.N.D.A.M.\ Gruppo Nazionale per le 
Strutture Algebriche, Geometriche e le loro Applicazioni.}
\eaddress{sandra.mantovani@unimi.it,\\ giuseppe.metere@unipa.it,\\ enrico.vitale@uclouvain.be}

\begin{document}

\maketitle


\begin{abstract}
We establish a generalized form both  of the Gabriel-Zisman exact sequence associated with a pointed functor between pointed groupoids, 
and of the Brown exact sequence associated with a fibration of pointed groupoids. Our generalization consists in replacing pointed 
groupoids with groupoids internal to a pointed regular category with reflexive coequalizers.
\end{abstract}

\tableofcontents

\section{Introduction}

One of the fundamental results stated in P. Gabriel and M. Zisman's book \cite{GZ} on categories of fractions and homotopy theory is the construction of a six terms exact sequence from a pointed functor between pointed groupoids. In order to obtain their exact sequence, Gabriel and Zisman make use of a special case of the comma square, and more precisely of what is sometimes called strong h-kernel (or strong homotopy fiber) of a pointed functor. Soon after Gabriel and Zisman's book, and certainly independently from Gabriel and Zisman, R.\ Brown described in \cite{RB} a six terms exact sequence associated with a fibration of pointed groupoids. Since Brown replaces arbitrary pointed functors with the more restrictive notion of fibration, he can use categorical kernels (or strict fibers) instead of strong h-kernels to construct his sequence. Moreover, the two results are logically equivalent: if the pointed functor is a fibration, the canonical comparison from the kernel to the strong h-kernel is an equivalence, so that the Gabriel-Zisman sequence reduces to the Brown sequence. Vice versa, any functor between groupoids can be (up to an equivalence)   turned into a fibration, and (in the pointed case) the kernel of the fibration is nothing but the strong h-kernel of the original functor.

The Gabriel-Zisman and Brown exact sequences have plenty of important applications, especially in algebraic topology and in 
non-abelian group homology. Since methods from internal category theory are currently used to study abelian and non-abelian homological algebra (see for example \cite{DB1, GJ, EG} and the references therein), our aim in this paper is to give a generalization of both Gabriel-Zisman and Brown exact sequences, replacing pointed groupoids with groupoids internal to a pointed regular category with reflexive coequalizers. More in detail, the layout of this paper is as follows. In Section 2 we review some basic facts on strong h-pullbacks and, in particular, on strong h-kernels in the 2-category $\Grpd(\cA)$ of groupoids internal to a suitable category $\cA.$ Section 3 is completely devoted to the construction of a six terms exact sequence in $\cA$ starting from an internal functor. The sequence involves the strong h-kernel of the internal functor, the connected components functor $\pi_0,$ and the automorphisms functor $\pi_1.$ In Section 4 we show that, if the internal functor is an internal fibration, we can replace the strong h-kernel with the categorical kernel. This fact is based on a result established in the companion paper \cite{JMMVFibr}, where fibrations of internal groupoids are studied more carefully. 
Section 5 shows how to get a (split epi) fibration from any internal functor. Finally, a simple application of the exact sequence to $\pi_0$ and $\pi_1$ is explained in Section 6.

If the base category $\cA$ is the category of groups, the exact sequence of Section 4 already appears in \cite{DKV} (and in \cite{KMV} 
as part of a ``ziqqurath'' of exact sequences). In fact, in this case 
the sequence is constructed from a monoidal functor, not just from an internal functor. Since monoidal functors between groupoids in 
groups are fractions of internal functors (see \cite{AN, FractEV}), one could wonder if the exact sequence can be still constructed starting
from a butterfly or a fractor (butterflies and fractors replace monoidal functors to describe fractions with respect to weak equivalences
of internal groupoids when the base category $\cA$ is semi-abelian, see \cite{AMMV}, or efficiently regular, see \cite{MMV}). The answer is positive,
but the proof makes use of the machinery of bicategories of fractions (\cite{Ben, Pronk}), and therefore we treat this problem in a separate 
paper \cite{JMMVFract}.

To end, an explication about terminology. If the base category $\cA$ is abelian (or semi-abelian), then via the normalization process which associates a commutative square (or a morphism of internal crossed modules) with an internal functor, the exact sequence \emph{\`a la} Brown and the exact sequence \emph{\`a la} Gabriel-Zisman coincide with the classical exact sequence of the snake lemma and with the more recent exact 
sequence of the snail lemma (see \cite{DB2, SnailEV, SnailZJEV}). This is why we adopt the names of snail lemma and snake lemma for our generalization of, respectively, the Gabriel-Zisman and Brown results.

\hfill

Note that in this paper, the composition of two arrows
$$\xymatrix{ \ar[r]^-{f} & \ar[r]^-{g} &}$$
will be denoted by $f \cdot g$.

%

\section{Preliminaries on 2-categories and internal groupoids}


We adopt the following definition of strong h-pullback in a 2-category (see \cite{GR} and \cite{JMMVFibr} 
for basic facts on (strong) h-pullbacks).

\begin{Definition}\label{DefStrongHpb}{\rm 
Let $F \colon \bA \to \bB$ and $G \colon \bC \to \bB$ be 1-cells in a 2-category $\cB$ with invertible 2-cells.
A strong h-pullback of $F$ and $G$ is a diagram of the form
$$\xymatrix{\bP \ar[r]^{G'} \ar[d]_{F'} & \bA \ar[d]^{F} \\
\bC \ar@{}[ru]^{\varphi}|{\Rightarrow} \ar[r]_{G} & \bB}$$
satisfying the following universal property :
\begin{enumerate}

\item For any diagram of the form 
$$\xymatrix{\bX \ar[r]^{H} \ar[d]_{K} & \bA \ar[d]^{F} \\
\bC \ar@{}[ru]^{\mu}|{\Rightarrow} \ar[r]_{G} & \bB}$$
there exists a unique 1-cell $T \colon \bX \to \bP$ such that $T \cdot G' = H, T \cdot F' = K$ and $T \cdot \varphi = \mu.$

\item Given 1-cells $L, M \colon \bX \rightrightarrows \bP$ and 2-cells $\alpha \colon L \cdot F' \Rightarrow  M \cdot F'$ and
$\beta \colon L \cdot G' \Rightarrow M \cdot G',$ if
$$\xymatrix{L \cdot F' \cdot G \ar@{=>}[r]^-{\alpha \cdot G} \ar@{=>}[d]_{L \cdot \varphi} & M \cdot F' \cdot G \ar@{=>}[d]^{M \cdot \varphi} \\
L \cdot G' \cdot F \ar@{=>}[r]_-{\beta \cdot F} & M \cdot G' \cdot F}$$
commutes, then there exists a unique 2-cell $\mu \colon L \Rightarrow M$ such that $\mu \cdot F' = \alpha$ and $\mu \cdot G' = \beta.$

\end{enumerate}
}\end{Definition}

\begin{Text}\label{TextGrpdA}{\rm
We assume now that $\cA$ is a category with finite limits and reflexive coequalizers. When needed, we shall tacitly assume that $\cA$ is pointed (it has an object 0 which is initial and terminal). We denote by $\Grpd(\cA)$ the 2-category of groupoids, functors and natural transformations 
internal to $\cA.$ The notation for a groupoid $\bB$ in $\cA$ is 
$$\bB=(\xymatrix{B_1 \times_{c,d}B_1 \ar[r]^-{m} & B_1 \ar@<0.8ex>[r]^{d} \ar@<-0.8ex>[r]_{c} & B_0 \ar[l]|{e}} \;,\;
\xymatrix{B_1 \ar[r]^{i} & B_1})$$
where
$$\xymatrix{B_1 \times_{c,d}B_1 \ar[r]^-{\pi_2} \ar[d]_{\pi_1} & B_1 \ar[d]^{d} \\
B_1 \ar[r]_{c} & B_0}$$
is a pullback. The notation for a natural transformation $\alpha \colon F \Rightarrow G \colon \bA \rightrightarrows \bB$ is
$$\xymatrix{A_1 \ar@<0,5ex>[rr]^-{F_1} \ar@<-0,5ex>[rr]_-{G_1} \ar@<-0,5ex>[d]_{d} \ar@<0,5ex>[d]^{c} & & 
B_1 \ar@<-0,5ex>[d]_{d} \ar@<0,5ex>[d]^{c} \\
A_0 \ar@<0,5ex>[rr]^-{F_0} \ar@<-0,5ex>[rr]_-{G_0} \ar[rru]^<<<<<<<<{\alpha} & & B_0}$$
}\end{Text}

\begin{Text}\label{TextPastingLemma}{\rm
From \cite{JMMVFibr}, recall the following simple fact that holds in $\Grpd(\cA).$ If the left-hand part of the following diagram is a pullback and the right-hand part is a strong h-pullback, then the total diagram (filled with the 2-cell $\widehat{H}\cdot\varphi$)  \marginpar{$\lll$ }
is a strong h-pullback
$$\xymatrix{\bD \times_{H,F'} \bP \ar[r]^-{\widehat H} \ar[d]_{\widehat{F'}} & \bP \ar[rr]^{G'} \ar[d]_{F'} & & \bA \ar[d]^{F} \\
\bD \ar[r]_-{H} & \bC \ar@{}[rru]^{\varphi}|{\Rightarrow} \ar[rr]_{G} & & \bB}$$
}\end{Text}

\begin{Text}\label{TextStrongHpbGrpd}{\rm
In \cite{JMMVFibr}, the previous point is used to show that the 2-category $\Grpd(\cA)$ has strong h-pullbacks. Indeed, a strong h-pullback
$$\xymatrix{\bP \ar[r]^{G'} \ar[d]_{F'} & \bA \ar[d]^{F} \\
\bC \ar@{}[ru]^{\varphi}|{\Rightarrow} \ar[r]_{G} & \bB}$$
can be constructed in two steps. First, one constructs the strong h-pullback
$$\xymatrix{\vec{\bB} \ar[r]^{\gamma} \ar[d]_{\delta} & \bB \ar[d]^{\Id} \\
\bB \ar@{}[ru]^{\beta}|{\Rightarrow} \ar[r]_{\Id} & \bB}
\;\; \mbox{ where }\;\;
\xymatrix{\vec B_1 \ar[r]^-{m_2} \ar[d]_{m_1} & B_1 \times_{c,d}B_1 \ar[d]^{m} \\
B_1 \times_{c,d}B_1 \ar[r]_-{m} & B_1}$$
is a pullback in $\cA,$ and
$\vec\bB=(\xymatrix{\vec B_1 \times_{\vec c,\vec d}\vec B_1 \ar[r]^-{\vec m} & 
\vec B_1 \ar@<0.9ex>[r]^{\vec d} \ar@<-0.9ex>[r]_{\vec c} & B_1 \ar[l]|{\vec e}} ,
\xymatrix{\vec B_1 \ar[r]^{\vec i} & \vec B_1})$ is the groupoid of commutative squares in $\bB$ (see also \cite{DR}).
The functors $\delta \colon \vec\bB \to \bB$ and $\gamma \colon \vec\bB \to \bB$ are given by
$$\xymatrix{\vec B_1 \ar[rr]^-{\delta_1=m_2 \cdot \pi_1} \ar@<-0.5ex>[d]_{\vec d} \ar@<0.5ex>[d]^{\vec c} & & 
B_1 \ar@<-0.5ex>[d]_{d} \ar@<0.5ex>[d]^{c} \\
B_1 \ar[rr]_-{\delta_0=d} & & B_0} \;\;\;\;\;
\xymatrix{\vec B_1 \ar[rr]^-{\gamma_1=m_1 \cdot \pi_2} \ar@<-0.5ex>[d]_{\vec d} \ar@<0.5ex>[d]^{\vec c} & & 
B_1 \ar@<-0.5ex>[d]_{d} \ar@<0.5ex>[d]^{c} \\
B_1 \ar[rr]_-{\gamma_0=c} & & B_0}$$
and the natural transformation $\beta \colon \delta \Rightarrow \gamma$ is simply $\beta = \id_{B_1} \colon B_1 \to B_1.$

Then, the desired strong h-pullback is given by the following limit diagram in $\Grpd(\cA)$
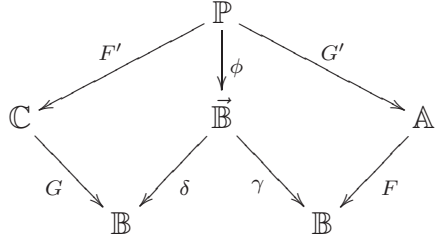
$$\xymatrix{& & \bP \ar[lld]_{F'} \ar[d]^{\phi} \ar[rrd]^{G'} \\
\bC \ar[rd]_{G} & & \vec\bB \ar[ld]^{\delta} \ar[rd]_{\gamma} & & \bA \ar[ld]^{F} \\
& \bB & & \bB}$$
together with $\varphi = \phi \cdot \beta \colon F' \cdot G = \phi \cdot \delta \Rightarrow \phi \cdot \gamma = G' \cdot F.$
Notice that in a category with pullbacks, such a limit can be obtained by means of two pullbacks.

}\end{Text}

\begin{Text}\label{TextStrongHpbGrpdDescr}{\rm
Since finite limits in $\Grpd(\cA)$ are constructed level-wise, the strong h-pullback $\bP$ can be described more explicitely as
the following limit diagram in $\cA$
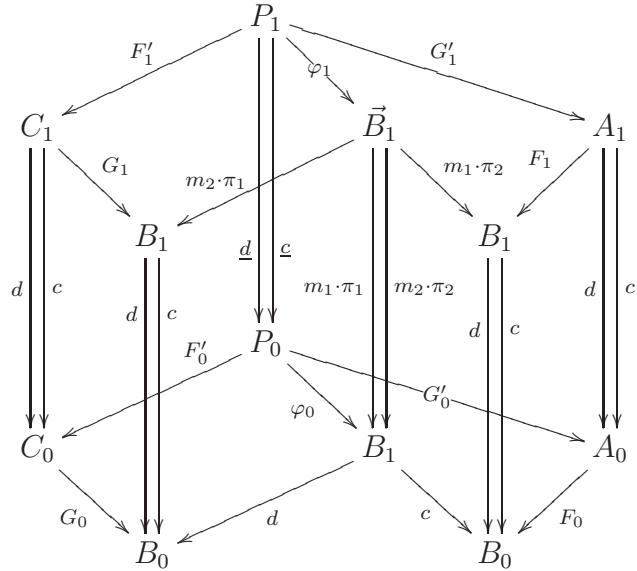
 $$\xymatrix{& & P_1 \ar[rd]|-{\varphi_1} \ar[rrrd]^-{G_1'} 
\ar@<-0.5ex>[ddd]_>>>>>>>>>>{\underline d} \ar@<0.5ex>[ddd]^>>>>>>>>>>{\underline c} \ar[lld]_-{F_1'} \\
C_1 \ar[rd]^{G_1} \ar@<-0.5ex>[ddd]_{d} \ar@<0.5ex>[ddd]^{c} & & & \vec B_1 \ar[lld]_>>>>>>>>>{m_2 \cdot \pi_1} 
\ar@<-0.5ex>[ddd]_{m_1 \cdot \pi_1} \ar@<0.5ex>[ddd]^{m_2 \cdot \pi_2} \ar[rd]^{m_1 \cdot \pi_2} 
& & A_1 \ar[ld]_<<<<<<{F_1} \ar@<-0.5ex>[ddd]_{d}  \ar@<0.5ex>[ddd]^{c} \\
& B_1 \ar@<-0.5ex>[ddd]_<<<<<<<<{d} \ar@<0.5ex>[ddd]^<<<<<<<<{c} 
& & & B_1 \ar@<-0.5ex>[ddd]_<<<<<<<<<<{d} \ar@<0.5ex>[ddd]^<<<<<<<<<<{c} \\
& & P_0 \ar[lld]_<<<<<<{F_0'} \ar[rd]_{\varphi_0} \ar[rrrd]|-{G_0'} \\
C_0 \ar[rd]_{G_0} & & & B_1 \ar[lld]^{d} \ar[rd]_{c} & & A_0 \ar[ld]^{F_0} \\
& B_0 & & & B_0}$$
}\end{Text}

\begin{Text}\label{TextStrongHkerGrpd}{\rm
In particular, if $\cA$ is pointed, the strong h-kernel
$$\xymatrix{\bK(F) \ar[rr]^{K(F)} \ar[d]_{0} & & \bA \ar[d]^{F} \\
[0]_0 \ar@{}[rru]^{k(F)_0}|{\Rightarrow} \ar[rr]_{0} & & \bB}$$
of a functor $F \colon \bA \to \bB$ exists in $\Grpd(\cA),$ and it can be explicitely described as the following limit diagram in $\cA$
$$\xymatrix{& & \bK(F)_1 \ar[rd]|-{k(F)_1} \ar[rrrd]^-{K(F)_1} 
\ar@<-0.5ex>[ddd]_>>>>>>>>>>{\underline d} \ar@<0.5ex>[ddd]^>>>>>>>>>>{\underline c} \ar[lld]_-{0} \\
0 \ar[rd]^{0} \ar@<-0.5ex>[ddd]_{0} \ar@<0.5ex>[ddd]^{0} & & & \vec B_1 \ar[lld]_>>>>>>>>>{m_2 \cdot \pi_1} 
\ar@<-0.5ex>[ddd]_{m_1 \cdot \pi_1} \ar@<0.5ex>[ddd]^{m_2 \cdot \pi_2} \ar[rd]^{m_1 \cdot \pi_2} 
& & A_1 \ar[ld]_<<<<<<{F_1} \ar@<-0.5ex>[ddd]_{d}  \ar@<0.5ex>[ddd]^{c} \\
& B_1 \ar@<-0.5ex>[ddd]_<<<<<<<<{d} \ar@<0.5ex>[ddd]^<<<<<<<<{c} 
& & & B_1 \ar@<-0.5ex>[ddd]_<<<<<<<<<<{d} \ar@<0.5ex>[ddd]^<<<<<<<<<<{c} \\
& & \bK(F)_0 \ar[lld]_<<<<<<{0} \ar[rd]_{k(F)_0} \ar[rrrd]|-{K(F)_0} \\
0 \ar[rd]_{0} & & & B_1 \ar[lld]^{d} \ar[rd]_{c} & & A_0 \ar[ld]^{F_0} \\
& B_0 & & & B_0}$$
}\end{Text}

\begin{Text}\label{TextPi0Pi1}\label{epsi}{\rm
Associated with a groupoid $\bB$, we can construct:
\begin{enumerate}
\item the object $\pi_0(\bB)$ of connected components, given by the coequalizer
$$\xymatrix{B_1 \ar@<-0,5ex>[r]_-{c} \ar@<0,5ex>[r]^-{d} & B_0 \ar[r]^-{\eta_{\bB}} & \pi_0(\bB)}$$
\item\label{lon} the object $\pi_1(\bB)$, which  is the joint kernel ($\text{Ker}(c)\cap \text{Ker}(d)$) of the domain and the codomain, given by the limit
$$\xymatrix{& & \pi_1(\bB) \ar[lld]_{0} \ar[d]^{\epsilon_{\bB}} \ar[rrd]^{0} \\
0 \ar[rd]_{0} & & B_1 \ar[ld]^{d} \ar[rd]_{c} & & 0 \ar[ld]^{0} \\
& B_0 & & B_0}$$
\end{enumerate}
These constructions are 2-functorial
$$\pi_0 \colon \Grpd(\cA) \to \cA \;,\;\; \pi_1 \colon \Grpd(\cA) \to \Grp(\cA)$$
where $\cA$ and $\Grp(\cA)$ (the category of internal groups in $\cA$) are seen as 2-categories with only identity 2-cells.
To see that $\pi_1(\bB)$ is indeed an internal group, just use its universal property to get multiplication and inverse
$$\xymatrix{\pi_1(\bB) \times \pi_1(\bB) \ar[r] \ar[d]_{\epsilon_{\bB} \times \epsilon_{\bB}}
& \pi_1(\bB) \ar[d]^{\epsilon_{\bB}} \\
B_1 \times_{c,d}B_1 \ar[r]_-{m} & B_1} \;\;\;\;\;
\xymatrix{\pi_1(\bB) \ar[r] \ar[d]_{\epsilon_{\bB}} & \pi_1(\bB) \ar[d]^{\epsilon_{\bB}} \\
B_1 \ar[r]_{i} & B_1}$$

Moreover, we have two pseudo-adjunctions
$$\pi_0 \dashv [ - ]_0 \;\mbox{ with }\; [X]_0 = X \rightrightarrows X, \;\mbox{ and }\;
[ - ]_1 \dashv \pi_1 \;\mbox{ with }\; [H]_1 = H \rightrightarrows 0$$

}\end{Text}

\begin{proof}
Here we check only that $\pi_0(F)=\pi_0(G)$ and $\pi_1(F)=\pi_1(G)$ if there exists a natural transformation 
$\alpha \colon F \Rightarrow G,$ and we leave the rest of the proof to the reader. Consider the diagram
$$\xymatrix{\pi_1(\bA) \ar@<0,5ex>[r]^-{\pi_1(F)} \ar@<-0,5ex>[r]_-{\pi_1(G)} \ar[d]_{\epsilon_{\bA}} & \pi_1(\bB) \ar[d]^{\epsilon_{\bB}} \\
A_1 \ar@<0,5ex>[r]^{F_1} \ar@<-0,5ex>[r]_-{G_1}  \ar@<-0,5ex>[d]_{d} \ar@<0,5ex>[d]^{c} & B_1 \ar@<-0,5ex>[d]_{d} \ar@<0,5ex>[d]^{c} \\
A_0 \ar@<0,5ex>[r]^{F_0} \ar@<-0,5ex>[r]_-{G_0} \ar[d]_{\eta_{\bA}} \ar[ru]|{\alpha} & B_0 \ar[d]^{\eta_{\bB}} \\
\pi_0(\bA) \ar@<0,5ex>[r]^-{\pi_0(F)} \ar@<-0,5ex>[r]_-{\pi_0(G)} & \pi_0(\bB)}$$
We have $\pi_0(F) = \pi_0(G)$ because $\eta_{\bA}$ is an epimorphism and
$$\eta_{\bA} \cdot \pi_0(F) = F_0 \cdot \eta_{\bB} = \alpha \cdot d \cdot \eta_{\bB} = 
\alpha \cdot c \cdot \eta_{\bB} = G_0 \cdot \eta_{\bB} = \eta_{\bA} \cdot \pi_0(G)$$
We have $\pi_1(F) = \pi_1(G)$ because $\epsilon_{\bB}$ is a monomorphism and
$$\pi_1(F) \cdot \epsilon_{\bB} = \epsilon_{\bA} \cdot F_1 = \epsilon_{\bA} \cdot \langle F_1, c \cdot \alpha \rangle \cdot \pi_1 =
\langle \epsilon_{\bA} \cdot F_1, \epsilon_{\bA} \cdot c \cdot \alpha \rangle \cdot \pi_1 = \langle \epsilon_{\bA} \cdot F_1, 0 \rangle \cdot \pi_1 = $$
$$= \langle \epsilon_{\bA} \cdot F_1, 0 \rangle \cdot m = \langle \epsilon_{\bA} \cdot F_1 , \epsilon_{\bA} \cdot c \cdot \alpha \rangle \cdot m = 
\epsilon_{\bA} \cdot \langle F_1 , c \cdot \alpha \rangle \cdot m = \epsilon_{\bA} \cdot \langle d \cdot \alpha , G_1 \rangle \cdot m = $$
$$= \langle \epsilon_{\bA} \cdot d \cdot \alpha , \epsilon_{\bA} \cdot G_1 \rangle \cdot m = \langle 0 , \epsilon_{\bA} \cdot G_1 \rangle \cdot m =
\langle 0 , \epsilon_{\bA} \cdot G_1 \rangle \cdot \pi_2 = \epsilon_{\bA} \cdot G_1 = \pi_1(G) \cdot \epsilon_{\bB}$$
\end{proof}

\begin{Text}\label{TextFFWeakEqui}{\rm
Following once again \cite{JMMVFibr}, we consider the strong h-pullbacks 
$$\xymatrix{\vec{\bA} \ar[r]^{\gamma} \ar[d]_{\delta} & \bA \ar[d]^{\Id} \\
\bA \ar@{}[ru]^{\alpha}|{\Rightarrow} \ar[r]_{\Id} & \bA}
\;\;\;\;\;\;\;\;\;\;\;\;
\xymatrix{\bR(F) \ar[rr]^{\gamma(F)} \ar[d]_{\delta(F)} & & \bA \ar[d]^{F} \\
\bA \ar@{}[rru]^{\alpha(F)}|{\Rightarrow} \ar[rr]_{F} & & \bB}$$
and the unique functor $\partial(F) \colon \vec\bA \to \bR(F)$ such that
 $\partial(F) \cdot \delta(F) = \delta$, $\partial(F) \cdot \gamma(F) = \gamma$ and
$\partial(F) \cdot \alpha(F) = \alpha \cdot F.$ The 0-level of the functor $\partial(F)$ is the unique arrow
making commutative the following diagram
$$\xymatrix{ & A_1 \ar[rr]^-{\partial(F)_0} \ar[ldd]_{d} \ar[rdd]_<<<<<<{F_1} \ar[rrrdd]_>>>>>>>>>>>>>>>>>>>>{c}
& & A_0 \times_{F_0,d} B_1 \times_{c,F_0} A_0 \ar[llldd]^>>>>>>>>>>>>>>{\delta(F)_0} \ar[ldd]^>>>>>>>>{\alpha(F)_0} \ar[rdd]^{\gamma(F)_0} \\ \\
A_0 \ar[rd]_{F_0} & & B_1 \ar[ld]^{d} \ar[rd]_{c} & & A_0 \ar[ld]^{F_0} \\
& B_0 & & B_0}$$
Therefore, we can complete the definitions given in \cite{BP} as follows.
}\end{Text}

\begin{Definition}\label{DefFFWeakEq}{\rm
A functor $F \colon \bA \to \bB$ in $\Grpd(\cA)$ is:
\begin{enumerate}
\item faithful if $\partial(F)_0$ is a monomorphism,
\item full if $\partial(F)_0$ is a regular epimorphism,
\item essentially surjective (surjective) if in one (equivalently, in both) of the following diagrams, where the squares are pullbacks, 
the first row is a regular epimorphism (a split epimorphism)
$$\xymatrix{A_0 \times_{F_0,d}B_1 \ar[r]^-{\beta_d} \ar[d]_{\alpha_d} & B_1 \ar[r]^-{c} \ar[d]^{d} & B_0 \\
A_0 \ar[r]_-{F_0} & B_0}
\;\;\;\;\;\;\;
\xymatrix{A_0 \times_{F_0,c}B_1 \ar[r]^-{\beta_c} \ar[d]_{\alpha_c} & B_1 \ar[r]^-{d} \ar[d]^{c} & B_0 \\
A_0 \ar[r]_-{F_0} & B_0}$$
\item a weak equivalence if it is full and faithful, and essentially surjective,
\item an equivalence if it is full and faithful, and surjective.
\end{enumerate}
}\end{Definition}

\begin{Text}\label{TextDefEquivFFunct}{\rm
It is well-known (see \cite{BP, EKVdL, FractEV}) that $F \colon \bA \to \bB$ is full and faithful (that is, $\partial(F)_0$ is an 
isomorphism) or an equivalence if and only if it is fully faithful or an equivalence in the 2-categorical sense (that is, the induced hom-functors 
$$- \cdot F \colon \Grpd(\cA)(\bX,\bA) \to \Grpd(\cA)(\bX,\bB)$$
are full and faithful or equivalences in the usual sense). We adapt hereunder the proof of Lemma 4.2 in \cite{FractEV} to show that in 
fact $F$ is faithful if and only if the functors $- \cdot F$ are faithful in the usual sense.
}\end{Text}

\begin{proof} 
Assume first that $F$ is faithful and consider two natural transformations	
$$\xymatrix{\bX \ar@/^1pc/[rr]^-{H} \ar@/_1pc/[rr]_-{K} \ar@{}[rr]|{\alpha \, \Downarrow \;\;\; \Downarrow \, \beta} & & \bA}$$
such that $\alpha \cdot F = \beta \cdot F,$ that is, such that $\alpha \cdot F_1 = \beta \cdot F_1.$ Since
$$\alpha \cdot \partial(F)_0 \cdot \delta(F)_0 = \alpha \cdot d = H_0 = \beta \cdot d = \beta \cdot \partial(F)_0 \cdot \delta(F)_0$$
$$\alpha \cdot \partial(F)_0 \cdot \gamma(F)_0 = \alpha \cdot c = K_0 = \beta \cdot c = \beta \cdot \partial(F)_0 \cdot \gamma(F)_0$$
$$\alpha \cdot \partial(F)_0 \cdot \alpha(F)_0 = \alpha \cdot F_1 = \beta \cdot F_1 = \beta \cdot \partial(F)_0 \cdot \alpha(F)_0$$
we have $\alpha \cdot \partial(F)_0 = \beta \cdot \partial(F)_0.$ Since $\partial(F)_0$ is a monomorphism, we have $\alpha = \beta.$ \\
Conversely, consider two arrows $\alpha, \beta \colon X_0 \rightrightarrows A_1$ such that $\alpha \cdot \partial(F)_0 = \beta \cdot \partial(F)_0.$
Since
$$\alpha \cdot d = \alpha \cdot \partial(F)_0 \cdot \delta(F)_0 = \beta \cdot \partial(F)_0 \cdot \delta(F)_0 = \beta \cdot d \;,\;\;
\alpha \cdot c = \alpha \cdot \partial(F)_0 \cdot \gamma(F)_0 = \beta \cdot \partial(F)_0 \cdot \gamma(F)_0 = \beta \cdot c$$
we can see $\alpha$ and $\beta$ as natural transformations as follows
$$\xymatrix{X_0 \ar@<-0,5ex>[dd]_{\id} \ar@<0,5ex>[dd]^{\id} 
\ar@<0,5ex>[rrr]^-{\alpha \cdot d \cdot e = \beta \cdot d \cdot e}
\ar@<-0,5ex>[rrr]_-{\alpha \cdot c \cdot e = \beta \cdot c \cdot e} 
& & & A_1 \ar@<-0,5ex>[dd]_{d} \ar@<0,5ex>[dd]^{c} \ar[r]^-{F_1}
& B_1 \ar@<-0,5ex>[dd]_{d} \ar@<0,5ex>[dd]^{c} \\ \\
X_0 \ar@<0,5ex>[rrruu]^{\alpha} \ar@<-0,5ex>[rrruu]_{\beta}
\ar@<0,5ex>[rrr]^-{\alpha \cdot d  = \beta \cdot d}
\ar@<-0,5ex>[rrr]_-{\alpha \cdot c = \beta \cdot c} 
& & & A_0 \ar[r]_-{F_0} & B_0}$$
Moreover, $\alpha \cdot F_1 = \alpha \cdot \partial(F)_0 \cdot \alpha(F)_0 = \beta \cdot \partial(F)_0 \cdot \alpha(F)_0 = \beta \cdot F_1.$ This means that
$\alpha \cdot F = \beta \cdot F$ as natural transformations. Since the hom-functor $- \cdot F$ is faithful, we conclude 
that $\alpha = \beta.$

\end{proof}

\section{The snail lemma for internal groupoids}

In this section, $\cA$ is a pointed regular category with reflexive coequalizers. 
Recall that the exactness in $B$ of
$$\xymatrix{A \ar[r]^-{f} & B \ar[r]^-{g} & C}$$
means that $(f,g)$ is a complex, that is, $f \cdot g = 0,$ and the factorization of $f$ through the kernel of $g$ is a regular epimorphism.

\begin{Text}\label{TextConnecting}{\rm
Starting from a functor $F \colon \bA \to \bB$ between groupoids in $\cA,$ we are going to construct an exact sequence
$$\xymatrix{\pi_1(\bK(F)) \ar[rr]^-{\pi_1(K(F))} & &  \pi_1(\bA) \ar[r]^{\pi_1(F)} & \pi_1(\bB) \ar[r]^-{D} 
& \pi_0(\bK(F)) \ar[rr]^-{\pi_0(K(F))} & & \pi_0(\bA) \ar[r]^{\pi_0(F)} & \pi_0(\bB)}$$
As far as the connecting morphism $D$ is concerned, let us observe that, since $\epsilon_{\bB} \cdot d = 0$ and 
$\epsilon_{\bB} \cdot c =0$ (where $\epsilon_{\bB}$ is as in \ref{epsi}.\ref{lon}) there exists a unique morphism $\Delta \colon \pi_1(\bB) \to \bK(F)_0$ such that 
$\Delta \cdot k(F)_0 = \epsilon_{\bB}$ and $\Delta \cdot K(F)_0 = 0.$ Therefore, we can define $D$ as follows:
$$D \colon \xymatrix{\pi_1(\bB) \ar[r]^{\Delta} & \bK(F)_0 \ar[r]^{\eta_{\bK(F)}} & \pi_0(\bK(F))}$$
}\end{Text}

\begin{Lemma}\label{LemmaDeltaMono}
(With the previous notation.) The diagram
$$\xymatrix{\pi_1(\bB) \ar[r]^-{\Delta} & \bK(F)_0 \ar[r]^-{K(F)_0} & A_0}$$
is a kernel diagram. 
\end{Lemma}

\begin{proof}
Observe that $\Delta$ is a monomorphism because $\Delta \cdot k(F)_0 = \epsilon_{\bB}$ and 
$\epsilon_{\bB}$ is a monomorphism. Now the direct proof of the universal property is an easy exercise.
\end{proof}

\begin{Lemma}\label{LemmaCompl}
Let $F \colon \bA \to \bB$ be a functor between groupoids in $\cA,$ together with its strong h-kernel 
$K(F) \colon \bK(F) \to \bA.$ The sequence
$$\xymatrix{\pi_1(\bK(F)) \ar[rr]^-{\pi_1(K(F))} & &  \pi_1(\bA) \ar[r]^{\pi_1(F)} & \pi_1(\bB) \ar[r]^-{D} 
& \pi_0(\bK(F)) \ar[rr]^-{\pi_0(K(F))} & & \pi_0(\bA) \ar[r]^{\pi_0(F)} & \pi_0(\bB)}$$
is a complex.
\end{Lemma}

\begin{proof}

\hfill

$\bullet$ The composite $\pi_1(K(F)) \cdot \pi_1(F)$ is trivial:
since there is a natural transformation $k(F)_0 \colon 0 \Rightarrow K(F) \cdot F,$ by 
\ref{TextPi0Pi1} we get $0 = \pi_1(0) = \pi_1(K(F) \cdot F) = \pi_1(K(F)) \cdot \pi_1(F).$

$\bullet$ The composite $\pi_1(F) \cdot D$ is trivial. Let us consider
$$i_F = \langle 0, \epsilon_{\bA} \cdot F_1 \rangle \colon \pi_1(\bA) \to B_1 \times_{c,d} B_1
\;\mbox{ and }\;
\vec i_F = \langle i_F, i_F \rangle \colon \pi_1(\bA) \to \vec B_1$$
Since $\vec i_F \cdot m_2 \cdot \pi_1 = i_F \cdot \pi_1 = 0$ and
$\vec i_F \cdot m_1 \cdot \pi_2 = i_F \cdot \pi_2 = \epsilon_{\bA} \cdot F_1,$
there exists a unique $\lambda \colon \pi_1(\bA) \to \bK(F)_1$ such that
$\lambda \cdot k(F)_1 = \vec i_F$ and $\lambda \cdot K(F)_1 = \epsilon_{\bA}.$
Composing with the limit projections, we check now that $\lambda \cdot \underline d = 0 \colon$
$$\lambda \cdot \underline d \cdot k(F)_0 = \lambda \cdot k(F)_1 \cdot m_1 \cdot \pi_1 =
\vec i_F \cdot m_1 \cdot \pi_1 = i_F \cdot \pi_1 = 0$$
$$\lambda \cdot \underline d \cdot K(F)_0 = \lambda \cdot K(F)_1 \cdot d = \epsilon_{\bA} \cdot d = 0$$
Similarly, we check that $\lambda \cdot \underline c = \pi_1(F) \cdot \Delta \colon$
$$\lambda \cdot \underline c \cdot k(F)_0 = \lambda \cdot k(F)_1 \cdot m_2 \cdot \pi_2 =
\vec i_F \cdot m_2 \cdot \pi_2 = i_F \cdot \pi_2 = \epsilon_{\bA} \cdot F_1 = \pi_1(F) \cdot \epsilon_{\bB} =
\pi_1(F) \cdot \Delta \cdot k(F)_0$$
$$\lambda \cdot \underline c \cdot K(F)_0 = \lambda \cdot K(F)_1 \cdot c = \epsilon_{\bA} \cdot c =
0 = \pi_1(F) \cdot 0 = \pi_1(F) \cdot \Delta \cdot K(F)_0$$
Finally : $\pi_1(F) \cdot D = \pi_1(F) \cdot \Delta \cdot \eta_{\bK(F)} =
\lambda \cdot \underline c \cdot \eta_{\bK(F)} = \lambda \cdot \underline d \cdot \eta_{\bK(F)} =
0 \cdot \eta_{\bK(F)} = 0.$ 

$\bullet$ The composite $D \cdot \pi_0(K(F))$ is trivial. This is a direct calculation :
$$D \cdot \pi_0(K(F)) = \Delta \cdot \eta_{\bK(F)} \cdot \pi_0(K(F)) =
\Delta \cdot K(F)_0 \cdot \eta_{\bA} = 0 \cdot \eta_{\bA} = 0$$

$\bullet$ The composite $\pi_0(K(F)) \cdot \pi_0(F)$ is trivial. 
Since there is a natural transformation $k(F)_0 \colon 0 \Rightarrow K(F) \cdot F,$ by 
\ref{TextPi0Pi1} we get $0 = \pi_0(0) = \pi_0(K(F) \cdot F) = \pi_0(K(F)) \cdot \pi_0(F).$
\end{proof}

The following definition is the version for groupoids of Definition 2.2 in \cite{SnailEV}, see also Section 5 in \cite{DB2}.

\begin{Definition}\label{DefProperGroupoid}{\rm
A groupoid $\bB$ is proper if the factorization $\beta$ of the pair $(d,c)$ through the kernel pair
of $\eta_{\bB}$ is a regular epimorphism
$$\xymatrix{B_1 \ar@<-0,5ex>[rr]_-{c} \ar@<0,5ex>[rr]^-{d} \ar[rd]_{\beta} & & 
B_0 \ar[r]^{\eta_{\bB}} & \pi_0{\bB} \\
& R[\eta_{\bB}] \ar@<-0,5ex>[ru]_{r_c} \ar@<0,5ex>[ru]^{r_d}}$$
}\end{Definition}

\begin{Text}\label{TextRestrRegEpi}{\rm
Consider the diagram
$$\xymatrix{B_1 \ar@<-0,5ex>[d]_{d} \ar@<0,5ex>[d]^{c} & & \Ker(d) \ar[ll]_-{k_d} \ar[d]^{c'} \\
B_0 \ar[d]_{\eta_{\bB}} & & \Ker(\eta_{\bB}) \ar[ll]^-{k_{\eta_{\bB}}} \\
\pi_0(\bB)}$$
where $c'$ is the unique arrow such that $k_d \cdot c = c' \cdot k_{\eta_{\bB}}$
(such an arrow exists because $k_d \cdot c \cdot \eta_{\bB} = k_d \cdot d \cdot \eta_{\bB} =
0 \cdot \eta_{\bB} = 0$). Then the diagram
$$\xymatrix{B_1 \ar[d]_{\beta} & & \Ker(d) \ar[ll]_-{k_d} \ar[d]^{c'} \\
R[\eta_{\bB}] & & \Ker(\eta_{\bB}) \ar[ll]^-{\langle 0, k_{\eta_{\bB}} \rangle}}$$
where $\langle 0, k_{\eta_{\bB}} \rangle \cdot r_d = 0$ and 
$\langle 0, k_{\eta_{\bB}} \rangle \cdot r_c =  k_{\eta_{\bB}},$ is a pullback.
The proof is straightforward using that the pair $(r_d,r_c)$ is monomorphic.
Therefore, $c'$ is a regular epimorphism whenever the groupoid $\bB$ is proper. \\
In the above argument, the role of $d$ and $c$ can be inverted: $d'$ is the unique arrow 
such that $k_c \cdot d = d' \cdot k_{\eta_{\bB}},$ and the diagram on the right is a pullback
$$\xymatrix{B_1 \ar@<-0,5ex>[d]_{d} \ar@<0,5ex>[d]^{c} & & \Ker(c) \ar[ll]_-{k_c} \ar[d]^{d'} \\
B_0 & & \Ker(\eta_{\bB}) \ar[ll]^-{k_{\eta_{\bB}}}}
\;\;\;\;\;
\xymatrix{B_1 \ar[d]_{\beta} & & \Ker(c) \ar[ll]_-{k_c} \ar[d]^{d'} \\
R[\eta_{\bB}] & & \Ker(\eta_{\bB}) \ar[ll]^-{\langle k_{\eta_{\bB}}, 0 \rangle}}$$
and again we get that $d'$ is a regular epimorphism whenever the groupoid $\bB$ is proper.
}\end{Text}

\begin{Proposition}\label{PropNonLinSnail}
Let $F \colon \bA \to \bB$ be a functor between groupoids in $\cA,$ together with its strong h-kernel 
$K(F) \colon \bK(F) \to \bA.$ If $\bA, \bB$ and $\bK(F)$ are proper, then the sequence
$$\xymatrix{\pi_1(\bK(F)) \ar[rr]^-{\pi_1(K(F))} & &  \pi_1(\bA) \ar[r]^{\pi_1(F)} & \pi_1(\bB) \ar[r]^-{D} 
& \pi_0(\bK(F)) \ar[rr]^-{\pi_0(K(F))} & & \pi_0(\bA) \ar[r]^{\pi_0(F)} & \pi_0(\bB)}$$
is exact.
\end{Proposition}

\begin{proof}

\hfill

$\bullet$ Exactness in $\pi_1(\bA).$ This follows from the pseudo-adjunction $[-]_1 \dashv \pi_1,$ 
see \ref{TextPi0Pi1}. 

$\bullet$ Exactness in $\pi_1(\bB).$ We have to prove that the factorization $\sigma$ of $\pi_1(F)$ 
through the kernel of $D = \Delta \cdot \eta_{\bK(F)}$ is a regular epimorphism.
$$\xymatrix{\Ker(D) \ar[r]^-{k_D} & \pi_1(\bB) \ar[r]^-{\Delta} & 
\bK(F)_0 \ar[r]^-{\eta_{\bK(F)}} & \pi_0(\bK(F)) \\
& \pi_1(\bA) \ar[u]_{\pi_1(F)} \ar[lu]^{\sigma}}$$
Consider the factorization
$$\xymatrix{\bK(F)_1 \ar@<-0,5ex>[rr]_-{\underline c} \ar@<0,5ex>[rr]^-{\underline d} \ar[rd]_{\kappa} 
& & \bK(F)_0 \ar[r]^{\eta_{\bK(F)}} & \pi_0{\bK(F)} \\
& R[\eta_{\bK(F)}] \ar@<-0,5ex>[ru]_{r_{\underline c}} \ar@<0,5ex>[ru]^{r_{\underline d}}}$$
and the unique arrow $\underline c'$ such that the diagram
$$\xymatrix{\bK(F)_1 \ar[d]_{\underline c} & \Ker(\underline d) 
\ar[l]_-{k_{\underline d}} \ar[d]^{\underline c'} \\
\bK(F)_0 & \Ker(\eta_{\bK(F)}) \ar[l]^-{k_{\eta_{\bK(F)}}}}$$
commutes. Following \ref{TextRestrRegEpi}, the diagram
$$\xymatrix{\bK(F)_1 \ar[d]_{\kappa} & & \Ker(\underline d) \ar[ll]_-{k_{\underline d}} 
\ar[d]^{\underline c'} \\
R[\eta_{\bK(F)}] & & \Ker(\eta_{\bK(F)}) \ar[ll]^-{\langle 0, k_{\eta_{\bK(F)}} \rangle}}$$
is a pullback, so that $\underline c'$ is a regular epimorphism because $\kappa$ is a regular 
epimorphism ($\bK(F)$ is proper). We are going to construct a diagram
$$\xymatrix{\pi_1(\bA) \ar[d]_{\sigma} \ar[r]^-{\lambda'} & \Ker(\underline d) \ar[d]^{\underline c'} \\
\Ker(D) \ar[r]_-{\Delta'} & \Ker(\eta_{\bK(F)})}$$
and prove that it is a pullback, which implies that $\sigma$ is a regular epimorphism.
In order to construct $\Delta',$ observe that $k_D \cdot \Delta \cdot \eta_{\bK(F)} = k_D \cdot D = 0,$ 
so that there exists a unique arrow $\Delta' \colon \Ker(D) \to \Ker(\eta_{\bK(F)})$ such that
$\Delta' \cdot k_{\eta_{\bK(F)}} = k_D \cdot \Delta.$ Moreover, $\Delta'$ is a monomorphism
because $\Delta$ is a monomorphism (see Lemma \ref{LemmaDeltaMono}).
Consider the arrow $\lambda \colon \pi_1(\bA) \to \bK(F)_1$ constructed in the proof of 
Lemma \ref{LemmaCompl} in order to prove that $\pi_1(F) \cdot D = 0.$
We already know that $\lambda \cdot \underline d = 0,$ so that there exists a unique arrow
$\lambda' \colon \pi_1(\bA) \to \Ker(\underline d)$ such that $\lambda' \cdot k_{\underline d} = \lambda.$
To check the commutativity of the above diagram, compose with the monomorphism 
$k_{\eta_{\bK(F)}}$ and recall that $\lambda \cdot \underline c = \pi_1(F) \cdot \Delta \colon$
$$\lambda' \cdot \underline c' \cdot k_{\eta_{\bK(F)}} = 
\lambda' \cdot \underline c' \cdot \langle 0, k_{\eta_{\bK(F)}} \rangle \cdot r_{\underline c} =
\lambda' \cdot k_{\underline d} \cdot \kappa \cdot r_{\underline c} =
\lambda \cdot \underline c = \pi_1(F) \cdot \Delta = \sigma \cdot k_D \cdot \Delta =
\sigma \cdot \Delta' \cdot k_{\eta_{\bK(F)}}$$
As far as the universality of the above diagram is concerned, consider two arrows
$$x \colon Z \to \Ker(D) \;\mbox{ and }\; y \colon Z \to \Ker(\underline d)$$ 
such that $x \cdot \Delta' = y \cdot \underline c'.$ In order to construct the factorization of 
$x$ and $y$ through $\sigma$ and $\lambda',$ we use the universal property of $\pi_1(\bA).$ 
Since
$$y \cdot k_{\underline d} \cdot K(F)_1 \cdot d = y \cdot k_{\underline d} \cdot \underline d \cdot K(F)_0 
= y \cdot 0 \cdot K(F)_0 = 0$$
$$y \cdot k_{\underline d} \cdot K(F)_1 \cdot c = y \cdot k_{\underline d} \cdot \underline c \cdot K(F)_0
= y \cdot k_{\underline d} \cdot \kappa \cdot r_{\underline c} \cdot K(F)_0 =
y \cdot \underline c' \cdot \langle 0, k_{\eta_{\bK(F)}} \rangle \cdot r_{\underline c} \cdot K(F)_0 =$$
$$= x \cdot \Delta' \cdot k_{\eta_{\bK(F)}} \cdot K(F)_0 =
x \cdot k_D \cdot \Delta \cdot K(F)_0 = x \cdot k_D \cdot 0 = 0$$
there exists a unique arrow $z \colon Z \to \pi_1(\bA)$ such that 
$z \cdot \epsilon_{\bA} = y \cdot k_{\underline d} \cdot K(F)_1.$
To check that $z \cdot \lambda' = y,$ compose with $k_{\underline d} \cdot K(F)_1,$ 
which is a monomorphism (this will be proved in Lemma \ref{LemmaKerFibrDiscr}):
$$y \cdot k_{\underline d} \cdot K(F)_1 = z \cdot \epsilon_{\bA} = z \cdot \lambda \cdot K(F)_1 =
z \cdot \lambda' \cdot k_{\underline d} \cdot K(F)_1$$
To check that $z \cdot \sigma = x,$ compose with the monomorphism $\Delta' \colon$
$$z \cdot \sigma \cdot \Delta' = z \cdot \lambda' \cdot \underline c' = y \cdot \underline c' = x \cdot \Delta'$$
Finally, such a factorization $z$ is necessarily unique. Indeed, $\lambda$ is a monomorphism 
(because $\epsilon_{\bA}$ is a monomorphism and $\lambda \cdot K(F)_1 = \epsilon_{\bA}$) and 
therefore $\lambda'$ also is a monomorphism because $\lambda' \cdot k_{\underline d} = \lambda.$ 

$\bullet$ Exactness in $\pi_0(\bK(F)).$  We have to prove that the factorization $\sigma$ of 
$D=\Delta \cdot \eta_{\bK(F)}$ through the kernel of $\pi_0(K(F))$ is a regular epimorphism.
$$\xymatrix{\Ker(\pi_0(K(F))) \ar[rr]^-{k_{\pi_0(K(F))}} & & \pi_0(\bK(F)) \ar[rr]^-{\pi_0(K(F))} & & \pi_0(\bA) \\
& & \pi_1(\bB) \ar[u]_{D} \ar[llu]^{\sigma}}$$
We are going to use the following diagram
$$\xymatrix{A_1 \ar[rdd]_{d} & \Ker(c) \ar[l]_-{k_c} \ar[d]^{d'} & & E \ar[ll]_-{f_0'} \ar[d]_{d''} \ar[r]^-{\Lambda}
& \pi_1(\bB) \ar@{-->}[ldd]^{\Delta} \ar[ddd]^{\sigma} \\
& \Ker(\eta_{\bA}) \ar[d]^{k_{\eta_{\bA}}} & & \Ker(K(F)_0 \cdot \eta_{\bA}) \ar[ll]_-{f_0} 
\ar[d]_{k_{K(F)_0 \cdot \eta_{\bA}}} \ar[rdd]_{\Sigma} \\
& A_0 \ar[d]_{\eta_{\bA}} & & \bK(F)_0 \ar[ll]_-{K(F)_0} \ar[d]_{\eta_{\bK(F)}} \\
& \pi_0(\bA) & & \pi_0(\bK(F)) \ar[ll]^-{\pi_0(K(F))} & \Ker(\pi_0(K(F))) \ar[l]^-{k_{\pi_0(K(F))}}}$$
where $d'$ is as in \ref{TextRestrRegEpi}, $f_0$ is the unique arrow such that 
$f_0 \cdot k_{\eta_{\bA}} = k_{K(F)_0 \cdot \eta_{\bA}} \cdot K(F)_0,$ 
$\Sigma$ is the unique arrow such that 
$\Sigma \cdot k_{\pi_0(K(F))} = k_{K(F)_0 \cdot \eta_{\bA}} \cdot \eta_{\bK(F)},$
$E$ is the pullback of $f_0$ and $d',$ and $\Lambda$ is to be constructed. The arrow $\Delta$ is dashed
because $\Lambda \cdot \Delta \neq d'' \cdot \Sigma$ (all the rest of the diagram is commutative).
If we can construct an arrow $\Lambda$ in such a way that $\Lambda \cdot \sigma = d'' \cdot \Sigma,$
then in order  to prove that $\sigma$ is a regular epimorphism it suffices to observe that $d''$ is a regular epimorphism
(it is the pullback of $d'$ which is a regular epimorphism since $\bA$ is proper), and $\Sigma$ 
also is a regular epimorphism. For this last fact, an easy inspection of the following diagram shows that 
the left-hand square is a pullback
$$\xymatrix{\Ker(K(F)_0 \cdot \eta_{\bA}) \ar[d]_{\Sigma} \ar[rr]^-{k_{K(F)_0 \cdot \eta_{\bA}}} & & 
\bK(F)_0 \ar[r]^-{K(F)_0} \ar[d]_{\eta_{\bK(F)}} & A_0 \ar[r]^-{\eta_{\bA}} & \pi_0(\bA) \ar[d]^{\id} \\
\Ker(\pi_0(K(F))) \ar[rr]_-{k_{\pi_0(K(F))}} & & \pi_0(\bK(F)) \ar[rr]_-{\pi_0(K(F))} & & \pi_0(\bA)}$$
In order to construct $\Lambda,$ observe that
$$d'' \cdot k_{K(F)_0 \cdot \eta_{\bA}} \cdot k(F)_0 \cdot c = d'' \cdot k_{K(F)_0 \cdot \eta_{\bA}} \cdot K(F)_0 \cdot F_0 =
d'' \cdot f_0 \cdot k_{\eta_{\bA}} \cdot F_0 =$$
$$=  f_0' \cdot d' \cdot k_{\eta_{\bA}} \cdot F_0 
= f_0' \cdot k_c \cdot d \cdot F_0 = f_0' \cdot k_c \cdot F_1 \cdot d$$
so that there exists a unique arrow $\tau \colon E \to B_1 \times_{c,d}B_1$ such that
$\tau \cdot \pi_1 = d'' \cdot k_{K(F)_0 \cdot \eta_{\bA}} \cdot k(F)_0$ and $\tau \cdot \pi_2 = f_0' \cdot k_c \cdot F_1.$
Moreover, since
$$\tau \cdot m \cdot d = \tau \cdot \pi_1 \cdot d = d'' \cdot k_{K(F)_0 \cdot \eta_{\bA}} \cdot k(F)_0 \cdot d = 
d'' \cdot k_{K(F)_0 \cdot \eta_{\bA}} \cdot 0 = 0$$
$$\tau \cdot m \cdot c = \tau \cdot \pi_2 \cdot c = f_0' \cdot k_c \cdot F_1 \cdot c = f_0' \cdot k_c \cdot c \cdot F_0 = f_0' \cdot 0 \cdot F_0 = 0$$
there exists a unique arrow $\Lambda \colon E \to \pi_1(\bB)$ such that
$$\xymatrix{E \ar[r]^-{\Lambda} \ar[d]_{\tau} & \pi_1(\bB) \ar[d]^{\epsilon_{\bB}} \\
B_1 \times_{c,d}B_1 \ar[r]_-{m} & B_1}$$
commutes.
It remains to check the equation
$$d'' \cdot \Sigma = \Lambda \cdot \sigma$$
Composing with the monomorphism $k_{\pi_0(K(F))},$ this is equivalent to checking the equation
$$d'' \cdot k_{K(F)_0 \cdot \eta_{\bA}} \cdot \eta_{\bK(F)} = \Lambda \cdot \Delta \cdot \eta_{\bK(F)}$$
and, for doing this, we construct a factorization of the pair $(d'' \cdot k_{K(F)_0 \cdot \eta_{\bA}}, \Lambda \cdot \Delta)$ 
through the pair $(\underline d, \underline c).$ This is done in  three steps. First, since we already know that 
$\tau \cdot m \cdot d =0,$ we can consider the factorization $\langle 0, \tau \cdot m \rangle \colon E \to B_1 \times_{c,d}B_1.$
Second, since the zero-arrow $0 \colon E \to B_1$ can be decomposed as
$$\xymatrix{E \ar[r]^-{f_0'} & \Ker(c) \ar[r]^-{k_c} & A_1 \ar[r]^-{c} & A_0 \ar[r]^-{F_0} & B_0 \ar[r]^-{e} & B_1}$$
there exists a unique arrow $S \colon E \to \vec B_1$ such that $S \cdot m_1 = \tau$ and 
$S \cdot m_2 = \langle 0, \tau \cdot m \rangle.$ Third, since
$$S \cdot m_2 \cdot \pi_1 = \langle 0, \tau \cdot m \rangle \cdot \pi_1 = 0 
\;\mbox{ and }\;
S \cdot m_1 \cdot \pi_2 = \tau \cdot \pi_2 = f_0' \cdot k_c \cdot F_1$$
there exists a unique arrow $\overline S \colon E \to \bK(F)_1$ such that $\overline S \cdot k(F)_1 = S$ and 
$\overline S \cdot K(F)_1 = f_0' \cdot k_c.$ Now, composing with the limit projections, we check the commutativity of
$$\xymatrix{\bK(F)_1 \ar[r]^-{\underline d} & \bK(F)_0 \\
E \ar[u]^{\overline S} \ar[ru]_{d'' \cdot k_{K(F)_0 \cdot \eta_{\bA}}}}
\;\;\;\;\;\;\;\;\;\;
\xymatrix{\bK(F)_1 \ar[r]^-{\underline c} & \bK(F)_0 \\
E \ar[u]^{\overline S} \ar[ru]_{\Lambda \cdot \Delta}}$$
$$\overline S \cdot \underline d \cdot k(F)_0 = \overline S \cdot k(F)_1 \cdot m_1 \cdot \pi_1 = S \cdot m_1 \cdot \pi_1 = 
\tau \cdot \pi_1 = d'' \cdot k_{K(F) \cdot \eta_{\bA}} \cdot k(F)_0$$
$$\overline S \cdot \underline d \cdot K(F)_0 = \overline S \cdot K(F)_1 \cdot d = f_0' \cdot k_c \cdot d = 
f_0' \cdot d' \cdot k_{\eta_{\bA}} = d'' \cdot f_0 \cdot k_{\eta_{\bA}} = d'' \cdot k_{K(F)_0 \cdot \eta_{\bA}} \cdot K(F)_0$$
$$\overline S \cdot \underline c \cdot k(F)_0 = \overline S \cdot k(F)_1 \cdot m_2 \cdot \pi_2 = S \cdot m_2 \cdot \pi_2
= \tau \cdot m = \Lambda \cdot \epsilon_{\bB} = \Lambda \cdot \Delta \cdot k(F)_0$$
$$\overline S \cdot \underline c \cdot K(F)_0 = \overline S \cdot K(F)_1 \cdot c = f_0' \cdot k_c \cdot c = f_0' \cdot 0 = 0 = 
\Lambda \cdot 0 = \Lambda \cdot \Delta \cdot K(F)_0$$
Finally,
$$d'' \cdot k_{K(F)_0 \cdot \eta_{\bA}} \cdot \eta_{\bK(F)} = \overline S \cdot \underline d \cdot \eta_{\bK(F)} =
\overline S \cdot \underline c \cdot \eta_{\bK(F)} = \Lambda \cdot \Delta \cdot \eta_{\bK(F)}$$

$\bullet$ Exactness in $\pi_0(\bA).$ We have to prove that the factorization $\sigma$ of $\pi_0(K(F))$ through the kernel
of $\pi_0(F)$ is a regular epimorphism.
$$\xymatrix{\Ker(\pi_0(F)) \ar[r]^-{k_{\pi_0(F)}} & \pi_0(\bA) \ar[r]^-{\pi_0(F)} & \pi_0(\bB) \\
& \pi_0(\bK(F)) \ar[lu]^{\sigma} \ar[u]_{\pi_0(K(F))}}$$
Since $k(F)_0 \cdot d = 0,$ there exists a unique arrow $\tau \colon \bK(F)_0 \to \Ker(d)$ such that $\tau \cdot k_d = k(F)_0.$
Consider now the pullback
$$\xymatrix{T \ar[r]^-{F_0'} \ar[d]_{k'} & \Ker(\eta_{\bB}) \ar[d]^{k_{\eta_{\bB}}} \\
A_0 \ar[r]_-{F_0} & B_0}$$
and the arrow $c' \colon \Ker(d) \to \Ker(\eta_{\bB})$ as in \ref{TextRestrRegEpi}. Since
$$\tau \cdot c' \cdot k_{\eta_{\bB}} = \tau \cdot k_d \cdot c = k(F)_0 \cdot c = K(F)_0 \cdot F_0$$
there exists a unique arrow $\tau' \colon \bK(F)_0 \to T$ such that $\tau' \cdot k' = K(F)_0$ and 
$\tau' \cdot F_0' = \tau \cdot c'.$ Moreover,
$$k' \cdot \eta_{\bA} \cdot \pi_0(F) = k' \cdot F_0 \cdot \eta_{\bB} = 
F_0' \cdot k_{\eta_{\bB}} \cdot \eta_{\bB} = F_0' \cdot 0 = 0$$
so that there exists a unique arrow $\gamma \colon T \to \Ker(\pi_0(F))$ such that 
$\gamma \cdot k_{\pi_0(F)} = k' \cdot \eta_{\bA}.$ We get the following diagram
$$\xymatrix{\bK(F)_0 \ar[r]^-{\tau'} \ar[d]_{\eta_{\bK(F)}} & T \ar[d]^{\gamma} \\
\pi_0(\bK(F)) \ar[r]_-{\sigma} & \Ker(\pi_0(F))}$$
and we check that it commutes by composing with the monomorphism $k_{\pi_0(f)} \colon$
$$\tau' \cdot \gamma \cdot k_{\pi_0(F)} = \tau' \cdot k' \cdot \eta_{\bA} = K(F)_0 \cdot \eta_{\bA} =
\eta_{\bK(F)} \cdot \pi_0(K(F)) = \eta_{\bK(F)} \cdot \sigma \cdot k_{\pi_0(F)}$$
To conclude that $\sigma$ is a regular epimorphism, it remains to prove that $\tau'$ and $\gamma$
are regular epimorphisms. As far as $\gamma$ is concerned, consider the diagrams
$$\xymatrix{T \ar[r]^-{F_0'} \ar[d]_{k'} & \Ker(\eta_{\bB}) \ar[d]^{k_{\eta_{\bB}}} \ar[r]^-{!} & 0 \ar[d]^{!} \\
A_0 \ar[r]_-{F_0} \ar@{}[ru]|-{(1)} & B_0 \ar[r]_-{\eta_{\bB}} \ar@{}[ru]|-{(2)} & \pi_0(\bB)}
\;\;\;\;\;
\xymatrix{T \ar[d]_{k'} \ar[r]^-{\gamma} & \Ker(\pi_0(F)) \ar[d]^{k_{\pi_0(F)}} \ar[r]^-{!} & 0 \ar[d]^{!} \\
A_0 \ar[r]_-{\eta_{\bA}} \ar@{}[ru]|-{(3)} & \pi_0(\bA) \ar[r]_-{\pi_0(F)} \ar@{}[ru]|-{(4)} & \pi_0(\bB)}$$
Since (1) and (2) are pullbacks, so is (1)+(2), that is, $k'$ is a kernel of 
$F_0 \cdot \eta_{\bB} = \eta_{\bA} \cdot \pi_0(F).$ This means that (3)+(4) is a pullback and,
since (4) also is a pullback, we have that (3) is a pullback. This implies that $\gamma$ is a regular
epimorphism because $\eta_{\bA}$ is a regular epimorphism.
As far as $\tau'$ is concerned, consider the diagram
$$\xymatrix{\bK(F)_0 \ar[r]^-{\tau'} \ar[d]_{\tau} & T \ar[r]^-{k'} \ar[d]_{F_0'} & A_0 \ar[d]^{F_0} \\
\Ker(d) \ar[r]_-{c'} \ar@{}[ru]|-{(5)} & \Ker(\eta_{\bB}) \ar[r]_-{k_{\eta_{\bB}}} \ar@{}[ru]|-{(1)} & B_0}$$

Since, $c'\cdot k_{\eta_{\bB}}= k_d\cdot c$ and $\tau'\cdot k'=K(F)_0$, then 
(5)+(1) is a pullback.

Since (1) also is a pullback, we deduce that (5) is a pullback.
Therefore, $\tau'$ is a regular epimorphism because $c'$ is a regular epimorphism 
(see \ref{TextRestrRegEpi}).
\end{proof}

\begin{Lemma}\label{LemmaKerFibrDiscr}
Let $F \colon \bA \to \bB$ be a functor between groupoids in $\cA,$ together with its strong h-kernel 
$K(F) \colon \bK(F) \to \bA.$ In the commutative diagram
$$\xymatrix{\Ker(\underline d) \ar[r]^-{k_{\underline d}} \ar[d]_{\K_d(K(F))} &
\bK(F)_1 \ar[r]^-{\underline d} \ar[d]^{K(F)_1} & \bK(F)_0 \ar[d]^{K(F)_0} \\
\Ker(d) \ar[r]_-{k_d} & A_1 \ar[r]_-{d} & A_0}$$
the square on the right is a pullback. As a consequence, the arrow $\K_d(K(F))$ is an isomorphism.
\end{Lemma}

Using the terminology of Definition \ref{DefFibration}, this lemma means that $K(F) \colon \bK(F) \to \bA$
is a discrete fibration.

\begin{proof} We have to prove that the canonical factorization $\tau_d$ in the diagram 
$$\xymatrix{\bK(F)_1 \ar[rr]^{K(F)_1} \ar[rd]^{\tau_d}  \ar[dd]_{\underline d} & & A_1 \ar[dd]^{d} \\
& \bK(F)_0 \times_{K(F)_0,d}A_1 \ar[ru]^{\beta_d} \ar[ld]_{\alpha_d} \\
\bK(F)_0 \ar[rr]_{K(F)_0} & & A_0}$$
is an isomorphism. In order to construct an inverse for $\tau_d,$ observe that
$$\alpha_d \cdot k(F)_0 \cdot c = \alpha_d \cdot K(F)_0 \cdot F_0 = \beta_d \cdot d \cdot F_0 = \beta_d \cdot F_1 \cdot d$$
Therefore, there exists a unique arrow $x \colon \bK(F)_0 \times_{K(F)_0,d} A_1 \to B_1 \times_{c,d} B_1$ such that 
$x \cdot \pi_1 = \alpha_d \cdot k(F)_0$ and $x \cdot \pi_2 = \beta_d \cdot F_1.$ Moreover,
$$x \cdot m \cdot d = x \cdot \pi_1 \cdot d = \alpha_d \cdot k(F)_0 \cdot d = \alpha_d \cdot 0 = 0 = 0 \cdot c$$
so that there exists a unique arrow $y \colon \bK(F)_0 \times_{K(F)_0,d}A_1 \to B_1 \times_{c,d}B_1$ such that 
$y \cdot \pi_1 = 0$ and $y \cdot \pi_2 = x \cdot m.$ Now, since $y \cdot \pi_1 = 0,$ we have $y \cdot m = y \cdot \pi_2$
and then $y \cdot m = x \cdot m.$ Therefore, there exists a unique arrow $z \colon \bK(F)_0 \times_{K(F)_0,d}A_1 \to \vec B_1$
such that $z \cdot m_1 = x$ and $z \cdot m_2 = y.$ Finally, since
$$z \cdot m_2 \cdot \pi_1 = y \cdot \pi_1 = 0 \;\mbox{ and }\; z \cdot m_1 \cdot \pi_2 = x \cdot \pi_2 = \beta_d \cdot F_1$$
there exists a unique arrow $t \colon \bK(F)_0 \times_{K(F)_0,d}A_1 \to \bK(F)_1$ such that 
$t \cdot k(F)_1 = z$ and $t \cdot K(F)_1 = \beta_d.$ \\
It remains to prove that $\tau_d$ and $t$ realize an isomorphism, which can be done by composing with the various limit 
projections. The only non straightforward condition to check is the following one:
$$\tau_d \cdot t \cdot k(F)_1 \cdot m_2 \cdot \pi_2  =  \tau_d \cdot z \cdot m_2 \cdot \pi_2 = \tau_d \cdot y \cdot \pi_2 =$$
$$= \tau_d \cdot x \cdot m = k(F)_1 \cdot m_1 \cdot m = k(F)_1 \cdot m_2 \cdot m = k(F)_1 \cdot m_2 \cdot \pi_2$$ 
where in the fourth equality  $\tau_d \cdot x \cdot = k(F)_1 \cdot m_1 $ since $\tau_d \cdot x \cdot \pi_1 = k(F)_1 \cdot m_1 \cdot \pi_1$ and 
$\tau_d \cdot x \cdot \pi_2 = k(F)_1 \cdot m_1 \cdot \pi_2,$ and the last equality comes from $k(F)_1 \cdot m_2 \cdot \pi_1 = 0.$
\end{proof}

\begin{Text}\label{TextProperGroupoid}{\rm
G. Janelidze pointed out to us that the condition to be proper 
is always satisfied by an internal groupoid if the base category $\cA$ is exact, but not if $\cA$ is just regular.
Here is the argument when $\cA$ is exact: start with a groupoid $\bB$ and consider the
 (regular epi, jointly monic)-factorization of $d,c \colon B_1 \rightrightarrows B_0$
$$\xymatrix{& \underline B \ar@<-0,5ex>[rd]_{\underline c} \ar@<0,5ex>[rd]^{\underline d} \\
B_1 \ar[ru]^{\underline \beta} \ar@<-0,5ex>[rr]_-{c} \ar@<0,5ex>[rr]^-{d} & & 
B_0 \ar[r]^{\eta_{\bB}} & \pi_0{\bB}}$$
Since $\cA$ is regular and $\bB$ is a groupoid, the pair 
$\underline d, \underline c \colon \underline B \rightrightarrows B_0$ is an equivalence
relation. Moreover, since $\underline \beta$ is a regular epi, the coequalizer of 
$(\underline d, \underline c)$ is $\eta_{\bB}.$ Therefore, if $\cA$ is exact,
$\underline d, \underline c \colon \underline B \rightrightarrows B_0$ is the 
kernel pair of $\eta_{\bB}$ and we have done.
}\end{Text}

\section{The snake lemma for internal groupoids}

In this section $\cA$ is a pointed regular category with reflexive coequalizers. 

Let us recall the definition of fibration, split epi fibration and discrete fibration for internal groupoids 
(the name ``split epi fibration'' is not standard, see \cite{JMMVFibr}).

\begin{Definition}\label{DefFibration}{\rm
Consider a functor $F \colon \bA \to \bB$ between groupoids in $\cA,$ and the induced factorizations through the 
pullbacks as in the following diagrams
$$\xymatrix{A_1 \ar[rr]^{F_1} \ar[rd]^{\tau_d}  \ar[dd]_{d} & & B_1 \ar[dd]^{d} \\
& A_0 \times_{F_0,d}B_1 \ar[ru]^{\beta_d} \ar[ld]_{\alpha_d} \\
A_0 \ar[rr]_{F_0} & & B_0}
\;\;\;\;\;\;\;
\xymatrix{A_1 \ar[rr]^{F_1} \ar[rd]^{\tau_c}  \ar[dd]_{c} & & B_1 \ar[dd]^{c} \\
& A_0 \times_{F_0,c}B_1 \ar[ru]^{\beta_c} \ar[ld]_{\alpha_c} \\
A_0 \ar[rr]_{F_0} & & B_0}$$
\begin{enumerate}
\item $F$ is a fibration when $\tau_d$ (equivalently, $\tau_c$) is a regular epimorphism.
\item $F$ is a split epi fibration when $\tau_d$ (equivalently, $\tau_c$) is a split epimorphism.
\item $F$ is a discrete fibration when $\tau_d$ (equivalently, $\tau_c$) is an isomorphism.
\end{enumerate}
}\end{Definition}


\begin{Text}\label{TextWhyFibr}{\rm
Having in mind the snail and the snake lemma in protomodular categories (see \cite{DB2} or \cite{SnailEV}),
the fact that fibrations enter in the picture is not a surprise. Here is why: given a functor $F \colon \bA \to \bB,$ 
consider the induced arrow $\K_d(F)$ as in the following diagram
$$\xymatrix{\Ker(d) \ar[r]^-{k_d} \ar[d]_{\K_d(F)} & A_1 \ar[r]^{d} \ar[d]_{F_1} & A_0 \ar[d]^{F_0} \\
\Ker(d) \ar[r]_-{k_d} & B_1 \ar[r]_{d} & B_0}$$
Then the commutative diagram
$$\xymatrix{\Ker(d) \ar[r]^{\K_d(F)} \ar[d]_{k_d \cdot c} & \Ker(d) \ar[d]^{k_d \cdot c} \\
A_0 \ar[r]_{F_0} & B_0}$$
is the normalization of $F \colon \bA \to \bB$ and, if $\cA$ is protomodular, it can be taken as starting 
point to construct the snail or the snake sequence as in \cite{SnailEV} (the snail sequence if we have 
no conditions on $\K_d(F),$ the snake sequence if $\K_d(F)$ is a regular epimorphism). Moreover, in
\cite{EKVdL} the following facts have been proved (see also \cite{JMMVFibr}):
\begin{enumerate}
\item If $F$ is a fibration, then the induced arrow $\K_d(F)$ is a regular epimorphism.
\item If the category $\cA$ is protomodular and if $\K_d(F)$ is a regular epimorphism, 
then $F$ is a fibration.
\end{enumerate}
}\end{Text}

\begin{Text}\label{TextIntroPropFibrCompKer}{\rm
In the next proposition, proved in \cite{JMMVFibr}, we compare the strong h-kernel $\bK(F)$
with the kernel $\bKer(F)$ of a functor $F.$ The latter is just the componentwise kernel in $\cA \colon$
$$\xymatrix{\Ker(F_1) \ar@<-0,5ex>[d]_{\underline d} \ar@<0,5ex>[d]^{\underline c} \ar[r]^-{k_{F_1}} &
A_1 \ar@<-0,5ex>[d]_{d} \ar@<0,5ex>[d]^{c} \ar[r]^-{F_1} & B_1 \ar@<-0,5ex>[d]_{d} \ar@<0,5ex>[d]^{c} \\
\Ker(F_0) \ar[r]_-{k_{F_0}} & A_0 \ar[r]_-{F_0} & B_0}$$
The universal property of the strong h-kernel induces a comparison $J$ as in the diagram
$$\xymatrix{\bK(F) \ar[r]^-{K(F)} & \bA \ar[r]^-{F} & \bB \\
\bKer(F) \ar[u]^{J} \ar[ru]_{K_F} }$$
}\end{Text}

\begin{Proposition}\label{PropFibrCompKer}
Consider a functor $F \colon \bA \to \bB$ between groupoids in $\cA$ together with the comparison $J \colon \bKer(F) \to \bK(F).$
\begin{enumerate}
\item If $F$ is a fibration, then $J$ is a weak equivalence.
\item If $F$ is a split epi fibration, then $J$ is an equivalence.
\end{enumerate}
\end{Proposition}

We need a partial generalization of Proposition 6.5 of \cite{EKVdL}.

\begin{Lemma}\label{LemmaEqDebEquiOmot}
Consider a functor $F \colon \bA \to \bB$ between groupoids in $\cA.$
\begin{enumerate}
\item If $F$ is full, then $\pi_1(F) \colon \pi_1(\bA) \to \pi_1(\bB)$ is a regular epimorphism.
\item If $F$ is faithful, then $\pi_1(F) \colon \pi_1(\bA) \to \pi_1(\bB)$ is a monomorphism.
\item If $F$ is full and $\bB$ is proper, then $\pi_0(F) \colon \pi_0(\bA) \to \pi_0(\bB)$ is a monomorphism.
\item If $F$ is essentially surjective, then $\pi_0(F) \colon \pi_0(\bA) \to \pi_0(\bB)$ is a regular epimorphism.
\item If $\pi_0(F) \colon \pi_0(\bA) \to \pi_0(\bB)$ is a regular epimorphism and $\bB$ is proper, then $F$ is essentially surjective.
\end{enumerate}
\end{Lemma}

\begin{proof} 1 and 2. As in \ref{TextFFWeakEqui}, we write $\partial(F)_0  \colon A_1 \to A_0 \times_{F_0,d} B_1 \times_{c,F_0} A_0$ 
for the unique arrow such that $\partial(F)_0  \cdot \delta(F)_0 = d, \partial(F)_0  \cdot \pi_a = F_1, \partial(F)_0  \cdot \gamma(F)_0 = c.$ Consider also the unique arrow
$\varphi \colon \pi_1(\bB) \to A_0 \times_{F_0,d} B_1 \times_{c,F_0} A_0$ such that $\varphi \cdot \delta(F)_0 = 0, \varphi \cdot \pi_a = 
\epsilon_{\bB}, \varphi \cdot \gamma(F)_0 = 0.$ Such a $\varphi$ is a monomorphism because $\epsilon_{\bB}$ is. 
We are going to prove that the following diagram is a pullback
$$\xymatrix{\pi_1(\bA) \ar[d]_{\epsilon_{\bA}} \ar[r]^-{\pi_1(F)} & \pi_1(\bB) \ar[d]^{\varphi} \\
A_1 \ar[r]_-{\partial(F)_0 } \ar@{}[ru]|{(1)} & A_0 \times_{F_0,d} B_1 \times_{c,F_0} A_0}$$
This immediately implies that, if $F$ is full (that is, if $\partial(F)_0 $ is a regular epimorphism), then $\pi_1(F)$ is a regular epimorphism.
Moreover, since $\epsilon_{\bA}$ is a monomorphism, if $F$ is faithful (that is, if $\partial(F)_0 $ is a monomorphism), then $\pi_1(F)$ is 
a monomorphism. For the commutativity of (1), just compose with the limit projections $\delta(F)_0, \pi_a, \gamma(F)_0.$ For the universality
of (1), consider the comparison $s$ with the pullback
$$\xymatrix{\pi_1(\bA) \ar[rr]^-{\pi_1(F)} \ar[dd]_{\epsilon_{\bA}} \ar[rd]^{s} & & \pi_1(\bB) \ar[dd]^{\varphi} \\
& A_1 \times_{\partial(F)_0 , \varphi} \pi_1(\bB) \ar[ld]_{\varphi'} \ar[ru]^{f'} \\
A_1 \ar[rr]_-{\partial(F)_0 } & & A_0 \times_{F_0,d} B_1 \times_{c,F_0} A_0}$$
Since
$$\varphi' \cdot d = \varphi' \cdot \partial(F)_0  \cdot \delta(F)_0 = f' \cdot \varphi \cdot \delta(F)_0 = f' \cdot 0 = 0 \;,\;\;
\varphi' \cdot c = \varphi' \cdot \partial(F)_0  \cdot \gamma(F)_0 = f' \cdot \varphi \cdot \gamma(F)_0 = f' \cdot 0 = 0$$
there exists a unique arrow $t \colon A_1 \times_{\partial(F)_0 ,\varphi}\pi_1(\bB) \to \pi_1(\bA)$ such that $t \cdot \epsilon_{\bA} = \varphi'.$
Moreover, 
$$s \cdot t \cdot \epsilon_{\bA} = s \cdot \varphi' = \epsilon_{\bA} \;,\;\; t \cdot s \cdot \varphi' = t \cdot \epsilon_{\bA} = \varphi'$$
and then $s \cdot t = \id$ and $t \cdot s = \id$ respectively because $\epsilon_{\bA}$ and $\varphi'$ are monomorphisms. \\
3. Let $x,y \colon S \rightrightarrows \pi_0(\bA)$ be two arrows such that $x \cdot \pi_0(F) = y \cdot \pi_0(F).$
In order to prove that $x=y,$ consider the pullbacks

$$
\xymatrix@!=4ex{
&&\bar S\ar[dl]_{\bar x}\ar[dr]^{\bar y}
\\
&S_x\ar[dl]_{x'} \ar[dr]^{\eta_x}
&&S_y\ar[dl]_{\eta_y} \ar[dr]^{y'}
\\
A_0\ar[dr]_{\eta_{\bA}}
&&S\ar[dl]^{x}\ar[dr]_{y}
&&A_0\ar[dl]^{\eta_{\bA}}
\\
&\pi_0(\bA)&&\pi_0(\bA)
}
$$

and the factorization 
$$\xymatrix{B_1 \ar@<-0,5ex>[rr]_-{c} \ar@<0,5ex>[rr]^-{d} \ar[rd]_{\beta} & & 
B_0 \ar[r]^{\eta_{\bB}} & \pi_0{\bB} \\
& R[\eta_{\bB}] \ar@<-0,5ex>[ru]_{r_c} \ar@<0,5ex>[ru]^{r_d}}$$
Since
$$\overline x \cdot x' \cdot F_0 \cdot \eta_{\bB} = \overline x \cdot x' \cdot \eta_{\bA} \cdot \pi_0(F) = 
\overline x \cdot \eta_x \cdot x \cdot \pi_0(F) = $$
$$= \overline y \cdot \eta_y \cdot y \cdot \pi_0(F) = \overline y \cdot y' \cdot \eta_{\bA} \cdot \pi_0(F) =
\overline y \cdot y' \cdot F_0 \cdot \eta_{\bB}$$
there exists a unique arrow $s \colon \overline S \to R[\eta_{\bB}]$ such that
$s \cdot r_d = \overline x \cdot x' \cdot F_0$ and $s \cdot r_c = \overline y \cdot y' \cdot F_0.$
Now we can construct the pullback
$$\xymatrix{S' \ar[r]^-{s'} \ar[d]_{\beta'} & B_1 \ar[d]^{\beta} \\
\overline S \ar[r]_-{s} & R[\eta_{\bB}]}$$
and, since
$$\beta' \cdot \overline x \cdot x' \cdot F_0 = \beta' \cdot s \cdot r_d = s' \cdot \beta \cdot r_d = s' \cdot d \;,\;\;
\beta' \cdot \overline y \cdot y' \cdot F_0 = \beta' \cdot s \cdot r_c = s' \cdot \beta \cdot r_c = s' \cdot c$$
there exists a unique arrow $\sigma \colon S' \to A_0 \times_{F_0,d} B_1 \times_{c,F_0} A_0$ such that
$\sigma \cdot \delta(F)_0 = \beta' \cdot \overline x \cdot x',$ $\sigma \cdot \pi_a = s'$ and  
$\sigma \cdot \gamma(F)_0 = \beta' \cdot \overline y \cdot y'.$ We can construct one more pullback
$$\xymatrix{S'' \ar[r]^-{\sigma'} \ar[d]_{f'} & A_1 \ar[d]^{\partial(F)_0 } \\
S' \ar[r]_-{\sigma} & A_0 \times_{F_0,d} B_1 \times_{c,F_0} A_0}$$
Since $\partial(F)_0 $ and $\beta$ are by assumption regular epimorphisms, then $f'$ and $\beta'$ also are reguar epimorphisms.
Moreover, $\eta_x$ and $\eta_y$ are regular epimorphisms (because $\eta_{\bA}$ is a regular epimorphism) 
and therefore $\overline x$ and $\overline y$ also are regular epimorphisms. Finally, to check that $x=y$ it suffices to 
check that $f' \cdot \beta' \cdot \overline x \cdot \eta_x \cdot x = f' \cdot \beta' \cdot \overline y \cdot \eta_y \cdot y \colon$
$$f' \cdot \beta' \cdot \overline x \cdot \eta_x \cdot x = f' \cdot \beta' \cdot \overline x \cdot x' \cdot \eta_{\bA} =
f' \cdot \sigma \cdot \delta(F)_0 \cdot \eta_{\bA} = \sigma' \cdot \partial(F)_0  \cdot \delta(F)_0 \cdot \eta_{\bA} = \sigma' \cdot d \cdot \eta_{\bA} =$$
$$= \sigma' \cdot c \cdot \eta_{\bA} = \sigma' \cdot \partial(F)_0  \cdot \gamma(F)_0 \cdot \eta_{\bA} = f' \cdot \sigma \cdot \alpha(F)_0 \cdot \eta_{\bA}
= f' \cdot \beta' \cdot \overline y \cdot y' \cdot \eta_{\bA} = f' \cdot \beta' \cdot \overline y \cdot \eta_y \cdot y$$
4. Assume that $F$ is essentially surjective, that is, $\beta_d \cdot c$ is a regular epimorphism
$$\xymatrix{A_0 \times_{F_0,d}B_1 \ar[d]_{\alpha_d} \ar[r]^-{\beta_d} 
& B_1 \ar[d]^{d} \ar[r]^-{c} & B_0 \\
A_0 \ar[r]_-{F_0} & B_0}$$
This implies that $\beta_d \cdot c \cdot \eta_{\bB}$ is a regular epimorphism. Moreover,
$$\beta_d \cdot c \cdot \eta_{\bB} = \beta_d \cdot d \cdot \eta_{\bB} =
\alpha_d \cdot F_0 \cdot \eta_{\bB} = \alpha_d \cdot \eta_{\bA} \cdot \pi_0(F)$$
so that $\pi_0(F)$ is a regular epimorphism. \\
5. Assume that $\pi_0(F)$ is a regular epimorphism. In the following pullback, $t_2$ is 
therefore a regular epimorphism
$$\xymatrix{A_0 \times_{\eta_{\bA} \cdot \pi_0(F), \eta_{\bB}}B_0 \ar[rr]^-{t_2}
\ar[d]_{t_1} & & B_0 \ar[d]^{\eta_{\bB}} \\
A_0 \ar[r]_{\eta_{\bA}} & \pi_0(\bA) \ar[r]_{\pi_0(F)} & \pi_0(\bB)}$$
Since
$$\beta_d \cdot c \cdot \eta_{\bB} = \beta_d \cdot d \cdot \eta_{\bB} =
\alpha_d \cdot F_0 \cdot \eta_{\bB} = \alpha_d \cdot \eta_{\bA} \cdot \pi_0(F)$$
there exists a unique arrow 
$t \colon A_0 \times_{F_0,d}B_1 \to A_0 \times_{\eta_{\bA} \cdot \pi_0(F), \eta_{\bB}}B_0$
such that $t \cdot t_1 = \alpha_d$ and $t \cdot t_2 = \beta_d \cdot c.$
To prove that $\beta_d \cdot c$ is a regular epimorphism, 
it remains to show that $t$ is a regular epimorphism.
For this, observe that, since $\eta_{\bA} \cdot \pi_0(F) = F_0 \cdot \eta_{\bB},$ 
the previous pullback can be split in two pullbacks
$$\xymatrix{A_0 \times_{\eta_{\bA} \cdot \pi_0(F), \eta_{\bB}}B_0 \ar[d]_{t_1}
\ar[r]^-{F_0'} & R[\eta_{\bB}] \ar[d]_{r_d} \ar[r]^-{r_c} & B_0 \ar[d]^{\eta_{\bB}} \\
A_0 \ar[r]_-{F_0} \ar@{}[ru]|{(1)} & B_0 \ar[r]_-{\eta_{\bB}} & \pi_0(\bB)}$$
with $F_0' \cdot r_c = t_2.$ Consider now the following diagram
$$\xymatrix{A_0 \times_{F_0,d}B_1 \ar[d]_{\beta_d} \ar[r]^-{t} &
A_0 \times_{\eta_{\bA} \cdot \pi_0(F), \eta_{\bB}}B_0 \ar[d]_{F_0'} \ar[r]^-{t_1}
& A_0 \ar[d]^{F_0} \\
B_1 \ar[r]_-{\beta} \ar@{}[ru]|{(2)} & R[\eta_{\bB}] \ar[r]_-{r_d} \ar@{}[ru]|{(1)} & B_0}$$
Composing with $r_d$ and $r_c,$ we check that (2) commutes:
$$t \cdot F_0' \cdot r_d = t \cdot t_1 \cdot F_0 =
\alpha_d \cdot F_0 = \beta_d \cdot d = \beta_d \cdot \beta \cdot r_d \;,\;\;
t \cdot F_0' \cdot r_c = t \cdot t_2 = \beta_d \cdot c = \beta_d \cdot \beta \cdot r_c$$
Finally, (2)+(1) is a pullback (because $t \cdot t_1 = \alpha_d$ and $\beta \cdot r_d = d$),
(1) is a pullback and (2) commutes, so that (2) is a pullback. This implies that $t$ is a regular
epimorphism because $\bB$ is proper.
\end{proof}

\begin{Proposition}\label{PropNonLinSnake}
Let $F \colon \bA \to \bB$ be a fibration between groupoids in $\cA.$ If $\bA, \bB$ and $\bK(F)$ 
are proper, then there exists an exact sequence
$$\xymatrix{\pi_1(\bKer(F)) \ar[rr]^-{\pi_1(K_F)} & &  \pi_1(\bA) \ar[r]^{\pi_1(F)} & \pi_1(\bB) \ar[r]
& \pi_0(\bKer(F)) \ar[rr]^-{\pi_0(K_F)} & & \pi_0(\bA) \ar[r]^{\pi_0(F)} & \pi_0(\bB)}$$
\end{Proposition}

\begin{proof}
Just consider the following commutative diagram
$$\xymatrix{\pi_1(\bK(F)) \ar[rr]^-{\pi_1(K(F))} & &  \pi_1(\bA) \ar[r]^{\pi_1(F)} & \pi_1(\bB) \ar[r]^-{D} 
& \pi_0(\bK(F)) \ar[rr]^-{\pi_0(K(F))} & & \pi_0(\bA) \ar[r]^{\pi_0(F)} & \pi_0(\bB) \\
\pi_1(\bKer(F)) \ar[u]^{\pi_1(J)} \ar[rru]_{\pi_1(K_F)} 
& & & & \pi_0(\bKer(F)) \ar[u]^{\pi_0(J)} \ar[rru]_{\pi_0(K_F)} }$$
By Proposition \ref{PropNonLinSnail}, the row is exact. By Proposition \ref{PropFibrCompKer}
and Lemma \ref{LemmaEqDebEquiOmot} applied to the comparison $J$
$$\xymatrix{\bK(F) \ar[r]^-{K(F)} & \bA \ar[r]^-{F} & \bB \\
\bKer(F) \ar[u]^{J} \ar[ru]_{K_F} }$$
the arrows $\pi_1(J)$ and $\pi_0(J)$ are isomorphisms.
\end{proof}

\section{Comparing the snake and the snail sequences}

In this section $\cA$ is a pointed regular category with reflexive coequalizers. 

In Section 4 we got the Snake sequence associated with a fibration as a special case 
of the Snail sequence associated with an arbitrary functor. In principle one can work in the opposite way.
This is because any functor between internal groupoids can be turned, up to an equivalence, into a fibration (in fact, a split epi fibration).

\begin{Proposition}\label{PropFunctFibr}
Let $F \colon \bA \to \bB$ be a functor between groupoids in $\cA.$ In the strong h-pullback
$$\xymatrix{\bF(F) \ar[rr]^-{F'} \ar[d]_{E} & & \bB \ar[d]^{\Id} \\
\bA \ar[rr]_-{F} & \ar@{}[u]|{f(F) \Downarrow} & \bB}$$
the functor $F'$ is a split epi fibration (and the functor $E$ is an equivalence).
\end{Proposition}

\begin{proof} 
Explicitly, the above strong h-pullback is
$$\xymatrix{& & \bF(F)_1 \ar[rd]|-{f(F)_1} \ar[rrrd]^-{E_1} 
\ar@<-0.5ex>[ddd]_>>>>>>>>>>{\underline d} \ar@<0.5ex>[ddd]^>>>>>>>>>>{\underline c} \ar[lld]_-{F'_1} \\
B_1 \ar[rd]^{\id} \ar@<-0.5ex>[ddd]_{d} \ar@<0.5ex>[ddd]^{c} & & & \vec B_1 \ar[lld]_>>>>>>>>>{m_2 \cdot \pi_1} 
\ar@<-0.5ex>[ddd]_{m_1 \cdot \pi_1} \ar@<0.5ex>[ddd]^{m_2 \cdot \pi_2} \ar[rd]^{m_1 \cdot \pi_2} 
& & A_1 \ar[ld]_<<<<<<{F_1} \ar@<-0.5ex>[ddd]_{d}  \ar@<0.5ex>[ddd]^{c} \\
& B_1 \ar@<-0.5ex>[ddd]_<<<<<<<<{d} \ar@<0.5ex>[ddd]^<<<<<<<<{c} 
& & & B_1 \ar@<-0.5ex>[ddd]_<<<<<<<<<<{d} \ar@<0.5ex>[ddd]^<<<<<<<<<<{c} \\
& & \bF(F)_0 \ar[lld]_<<<<<<{F'_0} \ar[rd]_{f(F)_0} \ar[rrrd]|-{E_0} \\
B_0 \ar[rd]_{\id} & & & B_1 \ar[lld]^{d} \ar[rd]_{c} & & A_0 \ar[ld]^{F_0} \\
& B_0 & & & B_0}$$
We have to prove that the factorization $\tau_c$ is a split epimorphism
$$\xymatrix{\bF(F)_1 \ar[rr]^{F'_1} \ar[dd]_{\underline c} \ar[rd]^{\tau_c} & & B_1 \ar[dd]^{c} \\
& \bF(F)_0 \times_{F'_0,c}B_1 \ar[ru]^{\beta_c} \ar[ld]_{\alpha_c} \\
\bF(F)_0 \ar[rr]_{F'_0} & & B_0}$$
To construct a section of $\tau_c$ we use the following three factorizations through pullbacks:
$$\xymatrix{\bF(F)_0 \times_{F'_0,c}B_1 \ar[r]^-{\alpha_c} \ar[rd]^{x} \ar[dd]_{\beta_c} 
& \bF(F)_0 \ar[r]^-{f(F)_0} & B_1 \ar[dd]^{d} \\
& B_1 \times_{c,d}B_1 \ar[ru]^{\pi_2} \ar[ld]_{\pi_1} \\
B_1 \ar[rr]_-{c} & & B_0}$$
indeed $\alpha_c \cdot f(F)_0 \cdot d = \alpha_c \cdot F'_0 = \beta_c \cdot c,$
$$\xymatrix{\bF(F)_0 \times_{F'_0,c}B_1 \ar[r]^-{\alpha_c} \ar[rd]^{y} \ar[dd]_{x} & 
\bF(F)_0 \ar[r]^-{E_0} & A_0 \ar[r]^-{F_0} & B_0 \ar[d]^{e} \\
& B_1 \times_{c,d}B_1 \ar[rr]^-{\pi_2} \ar[d]_{\pi_1} & & B_1 \ar[d]^{d} \\
B_1 \times_{c,d}B_1 \ar[r]_-{m} & B_1 \ar[rr]_-{c} & & B_0}$$
indeed $\alpha_c \cdot E_0 \cdot F_0 \cdot e \cdot d = \alpha_c \cdot E_0 \cdot F_0 = 
\alpha_c \cdot f(F)_0 \cdot c = x \cdot \pi_2 \cdot c = x \cdot m \cdot c,$
$$\xymatrix{\bF(F)_0 \times_{F'_0,c}B_1\ar[rr]^-{x} \ar[dd]_{y} \ar[rd]^{z} 
& & B_1 \times_{c,d}B_1 \ar[dd]^{m} \\
& \vec B_1 \ar[ru]^{m_2} \ar[ld]_{m_1} \\
B_1 \times_{c,d}B_1 \ar[rr]_-{m} & & B_1}$$
indeed $y \cdot \pi_2$ factors through $e \colon B_0 \to B_1,$ so that 
$y \cdot m = y \cdot \pi_1.$
Now observe that
$$z \cdot m_2 \cdot \pi_1 = x \cdot \pi_1 = \beta_c \,,\,\,
z \cdot m_1 \cdot \pi_2 = y \cdot \pi_2 = \alpha_c \cdot E_0 \cdot F_0 \cdot e = 
\alpha_c \cdot E_0 \cdot e \cdot F_1$$
so that, by the universal property of $\bF(F)_1,$ we get a unique arrow
$$\sigma_c \colon \bF(F)_0 \times_{F'_0,c}B_1 \to \bF(F)_1$$
such that $\sigma_c \cdot F'_1 = \beta_c, \sigma_c \cdot f(F)_1 = z, 
\sigma_c \cdot E_1 = \alpha_c \cdot E_0 \cdot e.$
It remains to check that $\sigma_c$ is a section of $\tau_c.$
Composing with the projections of the limit $\bF(F)_0,$ we have
$$\sigma_c \cdot \underline c \cdot F'_0 = \sigma_c \cdot F'_1 \cdot c = \beta_c \cdot c = 
\alpha_c \cdot F'_0 \,,\,\,
\sigma_c \cdot \underline c \cdot E_0 = \sigma_c \cdot E_1 \cdot c = 
\alpha_c \cdot E_0 \cdot e \cdot c = \alpha_c \cdot E_0,$$
$$\sigma_c \cdot \underline c \cdot f(F)_0 = \sigma_c \cdot f(F)_1 \cdot m_2 \cdot \pi_2 
= z \cdot m_2 \cdot \pi_2 = x \cdot \pi_2 = \alpha_c \cdot f(F)_0$$
so that $\sigma_c \cdot \underline c = \alpha_c.$ Finally, composing with the pullback
projections $\alpha_c$ and $\beta_c,$ we get
$$\sigma_c \cdot \tau_c \cdot \alpha_c = \sigma_c \cdot \underline c = \alpha_c \,,\,\,
\sigma_c \cdot \tau_c \cdot \beta_c = \sigma_c \cdot F'_1 = \beta_c$$
so that $\sigma_c \cdot \tau_c = \id.$
\end{proof}

\begin{Text}\label{TextFunctFibr}{\rm
The first part of the statement of Proposition \ref{PropFunctFibr} can be improved: 
for any strong h-pullback
$$\xymatrix{\bP \ar[r]^{G'} \ar[d]_{F'} & \bA \ar[d]^{F} \\
\bC \ar@{}[ru]^{\varphi}|{\Rightarrow} \ar[r]_{G} & \bB}$$
in $\Grpd(\cA),$ the functors $F'$ and $G'$ are split epi fibrations. The proof is a straightforward 
generalization of the proof of Proposition \ref{PropFunctFibr}.
}\end{Text}

\begin{Text}\label{TextComparingKerKK}{\rm
Consider again the strong h-pullback $\bF(F)$ used in Proposition \ref{PropFunctFibr} 
together with the strong h-kernels of  $F$ and $F'$ and the kernel of $F'$
$$\xymatrix{\bKer(F') \ar[rrd]^{K_{F'}} \ar[d]_{J} \\
\bK(F') \ar[rr]_-{K(F')} \ar[d]_{L} & & \bF(F) \ar[rr]^-{F'} \ar[d]_{E} & & \bB \ar[d]^{\Id} \\
\bK(F) \ar[rr]_{K(F)} & & \bA \ar[rr]_-{F} & \ar@{}[u]|{f(F) \Downarrow} & \bB}$$
Clearly, $L$ is an equivalence, and $J$ also is an equivalence because $F'$ is a split epi fibration 
(Proposition \ref{PropFibrCompKer}). Moreover, by \ref{TextPastingLemma} applied to 
the diagram
$$\xymatrix{\bKer(F') \ar[d]_{0} \ar[r]^-{K_{F'}} & \bF(F) \ar[d]_{F'} \ar[rr]^-{E} & & \bA \ar[d]^{F} \\
[0]_0 \ar[r]_-{0} & \bB \ar[rr]_-{\Id} & \ar@{}[u]|{f(F) \, \Rightarrow} & \bB}$$
we get that the composite $J \cdot L \colon \bKer(F') \to \bK(F)$ is an isomorphism. \\
It remains to compare 
the Snail sequence associated with $F'$ with the Snail sequence associated with $F.$ 
As expected, they are isomorphic exact sequences: this is a special case of the 
naturality of the Snail sequence stated below.
}\end{Text}

\begin{Proposition}\label{PropNatSnail}
A diagram in $\Grpd(\cA)$ of the form
$$\xymatrix{\bA' \ar[d]_{E} \ar[rr]^-{F'} & & \bB' \ar[d]^{T} \\
\bA \ar[rr]_-{F} & \ar@{}[u]|{\Downarrow \varphi} & \bB}$$
induces a morphism of complexes
$$\xymatrix{\pi_1(\bK(F')) \ar[rr]^-{\pi_1(K(F'))} \ar[d]_{\pi_1(L)} & &  
\pi_1(\bA') \ar[r]^{\pi_1(F')} \ar[d]_{\pi_1(E)} & \pi_1(\bB') \ar[r]^-{D'} \ar[d]_{\pi_1(T)}
& \pi_0(\bK(F')) \ar[rr]^-{\pi_0(K(F'))} \ar[d]^{\pi_0(L)} & & 
\pi_0(\bA') \ar[r]^{\pi_0(F')} \ar[d]^{\pi_0(E)} & \pi_0(\bB') \ar[d]^{\pi_0(T)} \\
\pi_1(\bK(F)) \ar[rr]_-{\pi_1(K(F))} & &  \pi_1(\bA) \ar[r]_{\pi_1(F)} & 
\pi_1(\bB) \ar[r]_-{D} \ar@{}[ru]|{(*)}
& \pi_0(\bK(F)) \ar[rr]_-{\pi_0(K(F))} & & \pi_0(\bA) \ar[r]_{\pi_0(F)} & \pi_0(\bB)}$$
where $L \colon \bK(F') \to \bK(F)$ is the canonical comparison between the strong h-kernels.\\
In particular, if $T$ and $E$ (and then $L$) are equivalences, then the complexes 
associated with $F$ and $F'$ are isomorphic. (The same holds if $T, E$ and $L$ are
weak equivalences, assuming that the groupoids $\bA, \bB$ and $\bK(F)$ are proper.)
\end{Proposition}

\begin{proof}
The non obvious part is to prove the commutativity of the square $(*)$ 
(the other squares commute by functoriality of $\pi_1$ and $\pi_0).$ 
We need an explicit description of $L_0 \colon \bK(F')_0 \to \bK(F)_0 \colon$ since
$$K(F')_0 \cdot \varphi \cdot d = K(F')_0 \cdot F'_0 \cdot T_0 = k(F')_0 \cdot c \cdot T_0 =
k(F')_0 \cdot T_1 \cdot c,$$ 
we get the following factorization
$$\xymatrix{\bK(F')_0 \ar[r]^-{K(F')_0} \ar[rd]^{\overline \varphi} \ar[d]_{k(F')_0} 
& A'_0 \ar[rd]^{\varphi} \\
B'_1 \ar[rd]_{T_1} & B_1 \times_{c,d}B_1 \ar[r]^-{\pi_2} \ar[d]_{\pi_1} & B_1 \ar[d]^{d} \\
& B_1 \ar[r]_{c} & B_0}$$
Moreover, since
$$\overline \varphi \cdot m \cdot d = \overline \varphi \cdot \pi_1 \cdot d = 
k(F')_0 \cdot T_1 \cdot d = k(F')_0 \cdot d \cdot T_0 = 0 \cdot T_0 = 0$$
$$\overline \varphi \cdot m \cdot c = \overline \varphi \cdot \pi_2 \cdot c = K(F')_0 \cdot \varphi \cdot c 
= K(F')_0 \cdot E_0 \cdot F_0$$
the universal property of $\bK(F)_0$ gives a unique arrow $L_0 \colon \bK(F')_0 \to \bK(F)_0$
such that $L_0 \cdot K(F)_0 = K(F')_0 \cdot E_0$ and $L_0 \cdot k(F)_0 = \overline \varphi \cdot m.$
Now we can split diagram $(*)$ in two parts
$$\xymatrix{\pi_1(\bB') \ar[r]^-{\Delta'} \ar[d]_{\pi_1(T)} & \bK(F')_0 \ar[d]^{L_0} \ar[r]^-{\eta_{\bK(F')}} 
& \pi_0(\bK(F')) \ar[d]^{\pi_0(L)} \\
\pi_1(\bB) \ar[r]_-{\Delta} \ar@{}[ru]|{(1)} & \bK(F)_0 \ar[r]_-{\eta_{\bK(F)}} \ar@{}[ru]|{(2)} & \pi_0(\bK(F))}$$
with square $(2)$ commuting by definition of $\pi_0(L).$ As far as square $(1)$ is concerned, 
we compose with the limit projections $K(F)_0$ and $k(F)_0.$ 
Composing both paths with $K(F)_0$ we get 0 :
$$\pi_1(T) \cdot \Delta \cdot K(F)_0 = \pi_1(T) \cdot 0 = 0 = 0 \cdot E_0 = 
\Delta' \cdot K(F')_0 \cdot E_0 = \Delta' \cdot L_0 \cdot K(F)_0$$
Composing with $k(F)_0$ we get
$$\pi_1(T) \cdot \Delta \cdot k(F)_0 = \pi_1(T) \cdot \epsilon_{\bB} = \epsilon_{\bB'} \cdot T_1 =
\Delta' \cdot k(F')_0 \cdot T_1 = \Delta' \cdot \overline \varphi \cdot \pi_1 = 
\Delta' \cdot \overline \varphi \cdot m = 
\Delta' \cdot L_0 \cdot k(F)_0$$
where the equality $\Delta' \cdot \overline \varphi \cdot \pi_1 = 
\Delta' \cdot \overline \varphi \cdot m$ comes from the fact that 
$\Delta' \cdot \overline \varphi \cdot \pi_2 = \Delta' \cdot K(F')_0 \cdot \varphi = 
0 \cdot \varphi = 0.$
\end{proof}

\section{The 2-functors $\pi_0$ and $\pi_1$ preserve exactness}

In this section, $\cA$ is a pointed regular category with reflexive coequalizers.

As an application of the snail lemma, in this section we prove that the 2-functors
$$\pi_0 \colon \Grpd(\cA) \to \cA \;,\;\; \pi_1 \colon \Grpd(\cA) \to \Grp(\cA)$$
introduced in Section 3 preserve exact sequences. The notion of exactness for a complex of internal functors is 
inspired by the notion of exactness in the 2-category of categorical groups introduced in \cite{PicBrEV}.

\begin{Definition}\label{Def2Exactness}{\rm 
Consider the following diagram in $\Grpd(\cA) \colon$
$$\xymatrix{& \bA \ar[rr]^-{0} \ar[rd]_{F} \ar[ld]_{F'} & & \bC \\
\bK(G) \ar[rr]_-{K(G)} & & \bB \ar[ru]_{G} \ar@{}[u]|{\Downarrow \, \varphi}}$$
We say that the sequence $(F, \varphi, G)$ is exact if the canonical comparison $F' \colon \bA \to \bK(G)$
is full and essentially surjective.
}\end{Definition}

\begin{Lemma}\label{LemmaPiPresExact}
Consider a functor between groupoids in $\cA,$ together with its strong h-kernel
$$\xymatrix{\bK(G) \ar[r]^-{K(G)} & \bB \ar[r]^-{G} & \bC}$$
Consider also the canonical comparisons $g_0$ and $g_1$ with the kernels of $\pi_0(G)$ and $\pi_1(G),$ as in the following diagrams
$$\xymatrix{\pi_0(\bK(G)) \ar[rr]^-{\pi_0(K(G))} \ar[rrd]_{g_0} & & \pi_0(\bB) \ar[r]^-{\pi_0(G)} & \pi_0(\bC) \\
& & \Ker(\pi_0(G)) \ar[u]_{k_{\pi_0(G)}}}
\;\;\;
\xymatrix{\pi_1(\bK(G)) \ar[rr]^-{\pi_1(K(G))} \ar[rrd]_{g_1} & & \pi_1(\bB) \ar[r]^-{\pi_1(G)} & \pi_1(\bC) \\
& & \Ker(\pi_1(G)) \ar[u]_{k_{\pi_1(G)}}}$$
\begin{enumerate}
\item The arrow $g_1$ is an isomorphism.
\item If $\bC$ is proper, the arrow $g_0$ is a regular epimorphism.
\end{enumerate}
\end{Lemma}

\begin{proof}
1. This follows from the pseudo-adjunction $[-]_1 \dashv \pi_1$ of \ref{TextPi0Pi1}. \\
2. This is the last point of the exact sequence of the snail lemma (Proposition \ref{PropNonLinSnail}).
\end{proof}

\begin{Proposition}\label{PropPiPresExact}
Consider an exact sequence in $\Grpd(\cA)$
$$\xymatrix{ & \bB \ar@{}[d]|{\Downarrow \, \varphi} \ar[rd]^{G} \\
\bA \ar[ru]^{F} \ar[rr]_-{0} & & \bC}$$
\begin{enumerate}
\item The sequence $\xymatrix{\pi_1(\bA) \ar[r]^-{\pi_1(F)} & \pi_1(\bB) \ar[r]^-{\pi_1(G)} & \pi_1(\bC)}$ is exact.
\item If $\bC$ is proper, the sequence $\xymatrix{\pi_0(\bA) \ar[r]^-{\pi_0(F)} & \pi_0(\bB) \ar[r]^-{\pi_0(G)} & \pi_0(\bC)}$ is exact.
\end{enumerate}
\end{Proposition}

\begin{proof}
Consider the commutative diagrams
$$\xymatrix{\pi_0(\bA) \ar[r]^-{\pi_0(F)} \ar[d]_{\pi_0(F')} & \pi_0(\bB) \\
\pi_0(\bK(G)) \ar[r]_-{g_0} & \Ker(\pi_0(G)) \ar[u]_{k_{\pi_0(G)}}}
\;\;\;\;\;
\xymatrix{\pi_1(\bA) \ar[r]^-{\pi_1(F)} \ar[d]_{\pi_1(F')} & \pi_1(\bB) \\
\pi_1(\bK(G)) \ar[r]_-{g_1} & \Ker(\pi_1(G)) \ar[u]_{k_{\pi_1(G)}}}$$
By Lemma \ref{LemmaPiPresExact}, $g_0$ is a regular epimorphism and $g_1$ is an isomorphism. Moreover, 
since $F' \colon \bA \to \bK(G)$ is full and essentially surjective, by Lemma \ref{LemmaEqDebEquiOmot} the
arrows $\pi_0(F')$ and $\pi_1(F')$ are regular epimorphisms.
\end{proof}

\end{document}